\documentclass[11pt]{amsart}
\usepackage{amssymb}
\usepackage{amsmath}
\usepackage{amsfonts}
\usepackage{graphicx}

\usepackage[total={17cm,22cm},top=2cm, left=2.3cm, bottom=2cm]{geometry}

\usepackage{hyperref}
    \usepackage{aeguill}
    \usepackage{type1cm}

\theoremstyle{plain}
\newtheorem{thm}{Theorem}[section]

\newtheorem{corollary}[thm]{Corollary}

\newtheorem{definition}[thm]{Definition}
\newtheorem{example}[thm]{Example}

\newtheorem{lemma}[thm]{Lemma}

\newtheorem{proposition}[thm]{Proposition}
\newtheorem{remark}[thm]{Remark}

\newtheorem{theorem}[thm]{Theorem}
\numberwithin{equation}{section}
\newcommand{\N}{\mathbb{N}}
\newcommand{\Z}{\mathbb{Z}}

\newcommand{\R}{\mathbb{R}}

\newcommand{\Rn}{\mathbb{R}^n}
\newcommand{\Sn}{\mathbb{S}^{n-1}}

\DeclareMathOperator{\sop}{supp\,\!}
\DeclareMathOperator{\vol}{vol\,\!}

\DeclareMathOperator{\diam}{diam\,\!}

\DeclareMathOperator{\interior}{int\,\!}

\usepackage[usenames,dvipsnames]{color}

\usepackage[dvipsnames]{xcolor}

\begin{document}

\title[Smooth convex extensions of convex functions]{Smooth convex extensions of convex functions}
\author{Daniel Azagra}
\address{ICMAT (CSIC-UAM-UC3-UCM), Departamento de An{\'a}lisis Matem{\'a}tico,
Facultad Ciencias Matem{\'a}ticas, Universidad Complutense, 28040, Madrid, Spain }
\email{azagra@mat.ucm.es}

\author{Carlos Mudarra}
\address{ICMAT (CSIC-UAM-UC3-UCM), Calle Nicol\'as Cabrera 13-15.
28049 Madrid SPAIN}
\email{carlos.mudarra@icmat.es}

\date{April 22, 2016}


\thanks{C. Mudarra was supported by Programa Internacional de Doctorado de la Fundaci\'on La Caixa--Severo Ochoa. Both authors were partially supported by Grant MTM2012-3431}


\begin{abstract}
Let $C$ be a compact convex subset of $\R^n$, $f:C\to\R$ be a convex function, and $m\in\{1, 2, \ldots, \infty\}$. Assume that, along with $f$, we are given a family of polynomials satisfying Whitney's extension condition for $C^m$, and thus that there exists $F\in C^{m}(\R^n)$ such that $F=f$ on $C$. It is natural to ask for further (necessary and sufficient) conditions on this family of polynomials which ensure that $F$ can be taken to be convex as well. We give a satisfactory solution to this problem in the case $m=\infty$, and make some remarks about the case of finite $m\geq 2$. For a solution to a similar problem in the case $m=1$ (even for $C$ not necessarily convex), see \cite{AM, ALM, AM2}.
\end{abstract}

\maketitle

\section{Introduction and main results}

Let $C$ be a closed subset of $\R^n$, and $m\in\N$. The famous Whitney Extension Theorem \cite{Whitney} provides a necessary and sufficient condition $(W^m)$ for a function $f: C\to\R$ and a family of polynomials $P_y:\R^n\to\R$, of degree$(P_y)\leq m$ and such $P_y(y)=f(y)$ for every $y\in C$, to admit a $C^m$ extension $F$ to all of $\R^n$ such that $J^{m}_yF=P_{y}$ for each $y\in C$, where $J_y^{m}F$ denotes the Taylor polynomial of order $m$ of $F$ at $y$. Whitney's condition $(W^m)$ can be reformulated by saying that
\begin{equation}\label{Whitney condition with polynomials}
\lim_{\delta\to 0^{+}}\rho_m(K,\delta)=0 \textrm{ for each compact subset } K \textrm{ of } C,
\end{equation}
where we denote
$$
\rho_m(K, \delta)=\sup \left\lbrace \frac{\|D^{j}P_y(z)-D^{j}P_z(z)\|}{|y-z|^{m-j}} \, : \, j=0,\ldots, m, \:  y,z\in K, \: 0<|y-z|\leq \delta \right\rbrace.
$$
If this condition is met, then Whitney's theorem provides us with a function $F\in C^{m}(\R^n)$ such that $D^{j}F(y)=D^{j}P_{y}(y)$ for every $j=0,\ldots,m$ and $y\in C$; see \cite[Theorem 3.1.14, p. 225]{Federer} for instance. The converse is trivially true.

In the case $m=\infty$, Whitney's theorem states that if we are given a family of polynomials $\{P^{m}_{y}\}_{y\in C, m\in \N\cup\{0\}}$ such that $P^{m}_y(y)=f(y)$ and for every $k>j$ the polynomial $P^{j}_y$ is the Taylor polynomial of order $j$ at $y$ of the polynomial $P^{k}_y$ (let us call such a family {\em a compatible family of polynomials for $C^{\infty}$ extension} of a function $f$ defined on $C$), and if for each $m\in\N$ the subfamily
$\{P^{m}_{y}\}_{y\in C}$ satisfies Whitney's condition \eqref{Whitney condition with polynomials}, then there is a function $F\in C^{\infty}(\R^n)$ such that $P^{m}_y=J^{m}_y F$, for every $y\in C$ and $m\in\N$. Again, the converse is obviously true.

In recent years there has been great interest in solving Whitney-type extension problems for functions rather than jets, in constructing continuous linear extension operators with nearly optimal norms, and in extending these results to other spaces of functions such as Sobolev spaces, see \cite{Glaeser, BrudnyiShvartsman, BierstoneMilmanPawluka1, FeffermanAnnals2005, FeffermanAnnals2006, BierstoneMilmanPawluka2, FeffermanBullAMS2009, ShvartsmanAdvances2009, Israel, FeffermanIsraelLuli, ShvartsmanAdvances2014, BrudnyiBrudnyi} and the references therein.

Returning to Whitney's theorem, it is natural to wonder what further conditions (if any) on those families of polynomials would be necessary and sufficient to ensure that $F$ can be taken to be convex whenever $f$ is convex. Besides its basic character, one should expect that a solution to this problem would find interesting applications in problems of differential geometry (see \cite{GhomiJDG2001} and the references therein, and also \cite[Theorem 1.8]{AM}), and of partial differential equations (such as the Monge-Amp\`ere equations).

Let us begin by making a couple of general observations concerning solvability of our extension problem. 

Firstly, if $C$ is not assumed to be compact, it is known that our problem has a negative solution. Indeed, there exists an unbounded closed convex subset $C$ of $\R^2$ and a $C^{\infty}$ convex function $f:C\to\R$ which has no {\em continuous} convex extension to all of $\R^2$, see \cite[Example 4]{SchulzSchwartz}. A modification of this example, which we will present in Section \ref{sectioncounterexamples} below, shows that the obstruction persists even if we require that $f$ have a strictly positive Hessian on a neighbourhood of $C$ (such strongly convex functions $f$ have smooth convex extensions to small open neighborhoods of $C$, but no convex extensions to $\R^n$). See also \cite{BorweinMontesinosVanderwerff, VeselyZajicek}, which show that there are infinite-dimensional Banach spaces $X$, closed {\em subspaces} $E\subset X$ and continuous convex functions $f:E\to\R$ which have no continuous convex extensions to $X$.

Secondly, if we do not require that $C$ be convex, then the problem gets geometrically complicated, for the following reason. There are several possible, nonequivalent, definitions of convex functions defined on non-convex domains (see \cite{MinYan} for a study of three of them) but, no matter how one defines convexity of such functions, the problem cannot be solved just by adding further analytical conditions on the putative Taylor polynomials of $f$ and disregarding the global geometry of the graph of $f$. 

To see why this is so, let us consider the following example: take any four numbers $a, b, c, d\in\R$ with $a<b<0<c<d$, and define $C=\{a, b, 0, c, d\}$ and $f(x)=|x|$ for $x\in C$.
Since $C$ is a five-point set it is clear that, no matter what polynomials of degree up to $k\geq 1$ are chosen to be the differential data of $f$ on $C$, the function $f$ will satisfy Whitney's extension condition $(W^k)$ for every $k\in\N$. Hence there are many $C^{1}$ (even infinitely many $C^{\infty}$) functions $F$ with $F=f$ on $C$. But none of these $F$ can be convex on $\R$, because, as is easily checked, any convex extension $g$ of $f$ to $\R$ must satisfy $g(x)=|x|$ for every $x\in [a, d]$, and therefore $g$ cannot be differentiable at $0$.
This example also shows that the most general forms of the extension problem for smooth convex functions are different in nature from the classical Whitney extension theorem \cite{Whitney} (for jets) and from the Whitney extension problems (for functions) dealt with in the mentioned papers \cite{Glaeser, BrudnyiShvartsman, BierstoneMilmanPawluka1, FeffermanAnnals2005, FeffermanAnnals2006, BierstoneMilmanPawluka2, FeffermanBullAMS2009, ShvartsmanAdvances2009, Israel, FeffermanIsraelLuli, ShvartsmanAdvances2014}, which are all of a local character (meaning that if one can obtain local solutions to the problem in question, with a uniform control of the modulus of continuity of the derivatives, then one can also obtain global solutions).

Fortunately, there is evidence that the geometrical obstructions present in these examples no longer exist when $C$ is assumed to be compact and convex. In particular, it is clear that if $f$ is convex on a compact convex set $C$ and is $C^{1}$ on a neighbourhood of $C$ then 
\begin{equation}\label{definition of the minimal extension}
m(f)(x):=\max_{y\in C}\{f(y)+ \langle \nabla f(y), x-y\rangle\}
\end{equation} defines a Lipschitz, convex function on all of $\R^n$ which coincides with $f$ on $C$ (and that, in the case when $C$ has nonempty interior, $m(f)$ happens to be the minimal convex extension of $f$ to $\R^n$).

Therefore, at least in a first approach to the problem, it seems reasonable to assume that $C$ is convex and compact, which we will do in the rest of this paper\footnote{Nonetheless, in the special case $m=1$, even for not necessarily convex and compact $C$, we have found in \cite{AM, AM2} global geometrical conditions which, along with $(W^1)$, are necessary and sufficient for the existence of convex functions $F\in C^{1}(\R^n)$ such that $F=f$ on $C$.}, and ask ourselves if our extension problem can always be solved in this relatively simple case.
Extension problems related to the one we are dealing with have been considered by M. Ghomi \cite{GhomiPAMS2002} and by M. Yan \cite{MinYan}. A consequence of their results is that, under the assumptions that $m\geq 2$ and that $f$ has a strictly positive Hessian on the boundary $\partial C$, there always exists an $F\in C^{m}(\R^n)$ such that $F$ is convex and $F=f$ on $C$. See also \cite{GhomiJDG2001, GhomiBLMS2004, BucicovschiLebl} for related problems. 

Of course, strict positiveness of the Hessian is a very strong condition which is far from being necessary, and it would be desirable to get rid of this requirement altogether, if possible. However, some other assumptions must be made in its place, at least when $m\geq 3$, as already in one dimension there are examples of $C^{3}$ convex functions $g$ defined on compact intervals $I$ which cannot be extended to $C^3(J)$ convex functions for any open interval $J$ containing $I$. Such an example is $g(x)=x^{2}-x^{3}$ defined for $x\in I:=[0, \frac{1}{3}]$. This example obviously generalizes to arbitrary dimension $n$ by considering for instance
\begin{equation}\label{positive Hessian does not imply CWm}
f(x_1, \ldots, x_n)=x_{1}^{2}+\cdots +x_n^2 - x_{1}^{3}, \,\,\, (x_1, \ldots, x_n)\in B(0, 1/3).
\end{equation}
In particular these examples show that the condition $D^{2}f\geq 0$ on $C$ is not sufficient to ensure the existence of a convex function $F\in C^{m}(\R^n)$ such that $F=f$ on $C$ for any $m\geq 3$. Therefore, we should look for conditions on the derivatives of $f$ on $C$ (beyond $D^{2}f\geq 0$ on $C$) that are necessary and sufficient to guarantee that $f$ has a $C^m$ convex extension $F$ to all of $\R^n$.

Now, observe that any such function $F$ will satisfy that $D^{2}F(x)(v^2)\geq 0$ for every $x\in\R^n$, $v\in \mathbb{S}^{n-1}$, and therefore, if $m\geq 2$ is finite, the Taylor polynomial of the second derivative $D^{2}F$ at points $y\in C$ will also satisfy
\begin{align*}
 0\leq &D^2F(y+tw)(v^2) = \\
 & D^2F(y)(v^2)+ t\: D^3F(y)(w,v^2)+ \cdots + \frac{t^{m-2}}{(m-2)!} D^mF(y)(w^{m-2},v^2) +R_{m}(t, y, v, w),
\end{align*}
where
$$
\lim_{t\to 0^{+}}\frac{R_{m}(t,y,v,w)}{t^{m-2}}=0 \quad \textrm{uniformly on} \quad y\in C, w, v\in \mathbb{S}^{n-1}.
$$
Then we will also have
$$
\liminf_{t\to 0^{+}}\frac{1}{t^{m-2}}\left(D^2F(y)(v^2)+ \cdots + \frac{t^{m-2}}{(m-2)!} D^mF(y)(w^{m-2},v^2) \right)\geq 0
$$
uniformly on $y\in C, w, v\in \mathbb{S}^{n-1}$. This of course means that for every $\varepsilon>0$ there exists $t_{\varepsilon}>0$ such that
$$
D^2F(y)(v^2)+ t\: D^3F(y)(w,v^2)+ \cdots + \frac{t^{m-2}}{(m-2)!} D^mF(y)(w^{m-2},v^2) \geq -\varepsilon t^{m-2}
$$
for all $y\in C$, $v, w\in\mathbb{S}^{n-1}$, $0<t\leq t_{\varepsilon}$.
We will abbreviate this by saying that
$$
F \textrm{ satisfies condition } (CW^m) \textrm{ on } C.
$$
Therefore we obtain the following necessary condition for the solution of the convex $C^{m}$ extension problem.

\begin{definition}
Let $m\in\N$, $m\geq 2$. We will say that $f$, together with a family of polynomials $\{P^{m}_y\}_{y\in C}$ of degree up to $m$ such that $P^{m}_y(y)=f(y)$, satisfy the condition $(CW^m)$ provided that for every $\varepsilon>0$ there exists $t_{\varepsilon}>0$ such that
$$
D^2P^{m}_{y}(y)(v^2)+ t\: D^3P^{m}_{y}(y)(w,v^2)+ \cdots + \frac{t^{m-2}}{(m-2)!} D^{m}P^{m}_{y}(y)(w^{m-2},v^2) \geq -\varepsilon t^{m-2}
$$
for all $y\in C$, $v, w\in\mathbb{S}^{n-1}$, $0<t\leq t_{\varepsilon}$. 

We will also say that $f$ and $\{P^{m}_y\}_{y\in C}$ {\em satisfy $(CW^{m})$ with a strict inequality} if there are some $\eta>0$ and $t_0>0$ such that
$$
D^2P^{m}_{y}(y)(v^2)+ t\: D^3P^{m}_{y}(y)(w,v^2)+ \cdots + \frac{t^{m-2}}{(m-2)!} D^mP^{m}_{y}(y)(w^{m-2},v^2) \geq \eta t^{m-2}
$$
for all $y\in  C$, $v, w\in\mathbb{S}^{n-1}$, $0<t\leq t_{0}$.
\end{definition}

In the case that $C$ has nonempty interior, the polynomials $P_{y}^{m}$ are uniquely determined (even at the boundary points of $C$) by the values of $f$ on $C$, and the above condition may be reformulated as follows
$$
\liminf_{t\to 0^{+}}\frac{1}{t^{m-2}}\left(D^2f(y)(v^2)+ \cdots + \frac{t^{m-2}}{(m-2)!} D^mf(y)(w^{m-2},v^2) \right)\geq 0 \eqno(CW^{m})
$$
uniformly on $y\in C, w, v\in \mathbb{S}^{n-1}$, understanding that $D^{j}f$ denotes the derivative of order $j$ at $y$ of any $C^{m}$ extension of $f$ to $\R^n$.

One might then think that for our convex extension problem, by considering the relative interior of the convex compact set $C$, there would be no loss of generality in assuming that $C$ has nonempty interior (and therefore considering that $(CW^{m})$ holds only for $v,w$ in the linear span of the directions $y-y'$ with $y, y'\in C$). However, since we are looking for convex analogues of the classical Whitney's extension theorem (which deals with prescribing differential data as well as extending functions) such an approach would make us lose some valuable insight about the question as to what extent one can prescribe values and derivatives of convex functions on a given compact convex set with empty interior. Indeed, for a convex compact set $C$ with empty interior and a convex function $f:C\to\mathbb{R}$, there are infinitely many convex functions $F:C\to\mathbb{R}$ with very different derivatives on $C$ and such that $F=f$ on $C$. 

Let us look, for instance, at the extreme situation in which $C$ is a singleton, say $C=\{0\}$. One of our main results in this paper (see Theorem \ref{maintheorem Cinfty} below) implies that, for any given family of polynomials $P^{m}$ of degree up to $m$, $m\in\N$, such that $J^{k}_{0}P^{m}=P^{k}$ whenever $k\leq m$ and, for each $m\geq 2$,
$$
\liminf_{t\to 0^{+}}\frac{D^2P^{m}(0)(v^2)+  \cdots + \frac{t^{m-2}}{(m-2)!} D^{m}P^{m}(0)(w^{m-2},v^2)}{t^{m-2}} \geq 0
$$
uniformly on $|v|=|w|=1$, there exists a convex function $F$ of class $C^{\infty}$ such that the Taylor polynomial of $F$ at $0$ is $P^{m}$. Consequently, there are infinitely many degrees of freedom in prescribing derivatives of convex functions at a given point.

On the other hand, if $C$ is a convex compact set with nonempty interior (what is usually called a convex body) and $f:C\to\R$ is a convex function which has a (not necessarily convex) $C^m$ extension to an open neighbourhood of $C$, then it is clear that $f$ automatically satisfies $D^2 f(x)\geq 0$ for every $x\in \interior(C)$.  Conversely, if $f$ satisfies $D^2f(x)\geq 0$ for all $x$ in the open convex set $\textrm{int}(C)$, then $f$ is convex on $\textrm{int}(C)$, and by continuity we infer that $f$ is also convex on $C$. These observations show that if $C$ is a convex compact subset of $\R^n$ with nonempty interior, $m\in\N, m\geq 2$, and $\lbrace P_y^m \rbrace_{y\in C}$ is a family of polynomials of degree up to $m,$ then a necessary condition for the existence of a convex function $F$ of class $C^m(\R^n)$ with $J_y^m F = P_y^m$ for every $y \in C$ is that
$$
\lbrace P_y^m \rbrace_{y\in C} \textrm{ satisfies } (W^m) \textrm{ on } C \textrm{ and }(CW^m) \textrm{ on } \partial C, \textrm{ and } D^2 P^m_y(y)(v^2) \geq 0 \textrm{ for all } y \in \interior(C), v\in \Sn;
$$ 
or equivalently that
$$
\lbrace P_y^m \rbrace_{y\in C} \textrm{ satisfies } (W^m) \textrm{ on } C \textrm{ and } (CW^m) \textrm{ on } \partial C, \textrm{ and the function } C \ni y \mapsto P^m_y(y) \textrm{ is convex.}
$$
It is also easy to show the following.
\begin{remark}
{\em
Let $f$ and $\{P^{m+1}_y\}_{y\in C}$ satisfy $(W^{m+1})$ and $(CW^{m+1})$ for some $m\geq 2$. Then $f$ and $\{P^{m}_y\}_{y\in C}$ satisfy $(CW^m)$ too, where each $P^{m}_y$ is obtained from  $P^{m+1}_y$ by discarding its $(m+1)$-homogeneous terms.}
\end{remark}

Our first main result is as follows.

\begin{theorem}\label{maintheorem Cinfty}
Let $C$ be a compact convex subset of $\R^n$. Let $f:C\to\R$ be a function, and let $\{P^{m}_{y}\}_{y\in C, m\in \N}$ be a compatible family of polynomials for $C^{\infty}$ extension of $f$. Then $f$ has a convex, $C^{\infty}$  extension $F$ to all of $\R^n$, with $J^{m}_{y}F=P^{m}_{y}$ for every $y\in C$ and $m\in\N$, if and only if $\{P^{m}_{y}\}_{y\in C}$ satisfies $(W^{m})$ and $(CW^{m})$ on $C$, for every $m\in \N$, $m\geq 2$.

Moreover, if $C$ has nonempty interior and $f:C\to\R$ is convex, then $f$ has a convex, $C^{\infty}$  extension $F$ to all of $\R^n$, with $J^{m}_{y}F=P^{m}_{y}$ for every $y\in C$ and $m\in\N$, if and only if $\{P^{m}_{y}\}_{y\in C}$ satisfies $(W^{m})$ on $C$, and $(CW^{m})$ on $\partial C$, for every $m\in \N$, $m\geq 2$.
\end{theorem}

If $m\geq 2$ finite and $n\geq 2$, if $C$ has empty interior then conditions $(CW^m)$ and $(W^m)$ are not sufficient for a convex function $f:C\to\R$ to have a $C^m$ convex extension to $\R^n$; see Example \ref{CWm with m finite is not sufficient when intC is empty} in Section \ref{sectioncounterexamples} below. However, it is conceivable that these conditions might be sufficient in the case that $C$ has nonempty interior. As of now, we only know that in dimension $n=1$ this is indeed so, and moreover, since in this case the boundary of $C$ has only two points and there are only two directions in which to differentiate, the condition $(CW^m)$ can be very much simplified.

\begin{proposition}
Let $I$ be a closed interval in $\R$, and $m\in\N$ with $m\geq 2$. Let $f:I\to\R$ be a convex function of class $C^{m}$ in the interior of $I$, and assume that $f$ has one-sided  derivatives of order up to $m$, denoted by $f^{(k)}(a^{+})$ or $f^{(k)}(b^{-})$, at the extreme points of $I$. Then $f$ has a convex extension of class $C^{m}(\R)$ if and only if the first (if any) non-zero derivative  which occurs in the finite sequence $\{f^{(2)}(b^{-}), f^{(3)}(b^{-}), \ldots, f^{(m)}(b^{-})\}$ is positive and of even order, and similarly for $\{f^{(2)}(a^{+}), f^{(3)}(a^{+}), \ldots, f^{(m)}(a^{+})\}$.
\end{proposition}
\noindent The easy proof is left to the reader's care.

In the special case that condition $(CW^k)$ is satisfied with a strict inequality for some $k$, the problem also becomes much easier to solve, because in this situation $f$ must be convex on a neighbourhood of $C$, and then we may use the following by-product of our proof of Theorem \ref{maintheorem Cinfty}

\begin{proposition}\label{if f is convex on a neighbourhood of C then the problem is easy}
Let $m\in\N$. If $C\subset\R^n$ is compact, and if there exists an open convex neighbourhood $U$ of $C$ such that $f:U\to\R$ is $C^m$ and convex, then there exists a convex function $F\in C^{m}(\R^n)$ such that $F=f$ on $C$.
\end{proposition}
See Section \ref{proof of corollary in case m finite} below for more information. Let us also note that in this case $f$ automatically satisfies $(CW^p)$ for all the rest of $p$'s.

\begin{proposition}
Let $m\in \N\cup\{\infty\}$, $m\geq 2$. If $f\in C^{m}(\R^n)$ satisfies $(CW^{k})$ with a strict inequality on $\partial C$ for some $k\geq 2$, then $f$ satisfies $(CW^p)$ on $\partial C$ for every $p\in\{2, \ldots, m\}$, if $m$ is finite, and for every $p\in\N$ with $p\geq 2$, if $m=\infty$.
\end{proposition}

We leave the easy verification to the reader's care. As a straightforward consequence of Proposition \ref{if f is convex on a neighbourhood of C then the problem is easy} we will obtain the following.

\begin{corollary}\label{corollary for Cm and CWk with a strict inequality}
Let $m\in \N\cup\{\infty\}$, $m\geq 2$. Let $C$ be a convex compact subset of $\R^n$, and let $f:C\to\R$ be a convex function having a (not necessarily convex) $C^{m}$ extension to an open neighbourhood of $C$. If $f$ and its derivatives satisfy $(CW^ k)$ with a strict inequality  on $C$ for some $2\leq k\leq m$, then there exists a convex function $F\in C^{m}(\R^n)$ such that $F=f$ on $C$.
\end{corollary}

The easiest instance of application of this corollary is of course when $f$ has a strictly positive Hessian on $\partial C$, in which case we recover the aforementioned consequence of the results of M. Ghomi's \cite{GhomiPAMS2002}  and M. Yan's \cite{MinYan}.

In the case $m\geq 2$ with $m$ finite, the method of proof of Theorem \ref{maintheorem Cinfty} does not allow us to obtain sufficiency of the conditions $(W^m)$ and $(CW^m)$ for a function $f$ to have a convex $C^m$ extension. The best we can obtain with this method is the following.

\begin{theorem}\label{maintheoremfinitecase}
Let $C$ be a compact convex subset of $\R^n$. Let $f:C\to\R$ be a function, $m\in\N$ with $m\geq n+3 $, and let $\{P^{m}_{y}\}_{y\in C}$ be a family of polynomials of degree less than or equal to $m$ and such that $P^{m}_y(y)=f(y)$ for every $y\in C$. Assume that $\{P^{m}_{y}\}_{y\in C}$ satisfies $(W^{m})$ and $(CW^{m})$. Then $f$ has a convex extension $F\in C^{m-n-1}(\R^n)$ such that $J^{m-n-1}_{y}F=P^{m-n-1}_{y}$ for every $y\in C$.
\end{theorem} 

The above result is probably not optimal, at least in the case when $C$ has nonempty interior. However, Example \ref{CWm with m finite is not sufficient when intC is empty} below will show that if $C$ has empty interior then one cannot expect to find smooth convex extensions of jets satisfying $(W^m)$ and $(CW^m)$ on $C$ without losing at least two orders of smoothness. On the positive side, there is a class of relatively nice convex bodies for which Theorem \ref{maintheoremfinitecase} can be very much improved.

\begin{definition}[FIO bodies of class $m$]\label{definitionscm3}
Given an integer $m \geq 2,$ we will say that a subset $C$ of $\Rn$ is an ovaloid of class $C^m$ if there exist $ M>0$ and a function $\psi: \Rn \to \R$ such that
\begin{enumerate}
\item[(i)] $\psi$ is of class $C^m(\Rn).$
\item[(ii)] $D^2 \psi (x)(v^2)  \geq M$ for all $x\in \Rn$ and for all $v\in \Sn.$
\item[(iii)] $C= \psi^{-1}(-\infty, 1].$
\end{enumerate}
We will also say that a set $C$ is $(FIO^m)$, or an FIO body of class $C^m$, if $C$ is the intersection of a finite family of ovaloids of class $C^m$.
\end{definition}

By restricting our attention to the class of FIO bodies, we can find convex extensions of functions satisfying $(W^m)$ and $(CW^m)$ with a loss of just one order of smoothness.

\begin{theorem}\label{theoremspecialcase}
Let $C$ be a convex subset of $\R^n.$ Let $f:C\to\R$ be a function, $m\in\N$ with $m\geq 3$, and let $\{P^{m}_{y}\}_{y\in C}$ be a family of polynomials of degree less than or equal to $m$ and such that $P^{m}_y(y)=f(y)$ for every $y\in C$. Assume that $\{P^{m}_{y}\}_{y\in C}$ satisfies $(W^{m})$ and $(CW^{m})$, and that $C$ is $(FIO^{m-1}).$ Then $f$ has a convex extension $F\in C^{m-1}(\R^n)$ such that $J^{m-1}_{y}F=P^{m-1}_{y}$ for every $y\in C$.
\end{theorem} 

Let us conclude this introduction with an important remark: one might wonder whether the conditions $(CW^m)$ could be deduced from the condition $D^2f\geq 0$ on $C$, at least in the case that $C$ has nonempty interior. The answer is negative: in view of Theorem \ref{theoremspecialcase} and the example given in equation \eqref{positive Hessian does not imply CWm} above, the condition $D^2f\geq 0$ on a convex body $C$ does not imply condition $(CW^m)$ on $C$ for any $m\geq 4$. Furthermore, by making some straightforward  calculations on can show that the function $f$ defined in \eqref{positive Hessian does not imply CWm} does not satisfy condition $(CW^3)$ either.  

The rest of the paper is organized as follows. In Section \ref{section proof cinfty} we will prove Theorem \ref{maintheorem Cinfty} and Corollary \ref{corollary for Cm and CWk with a strict inequality}. In Section \ref{sectionproofcm} we will prove Theorem \ref{maintheoremfinitecase} and Theorem \ref{theoremspecialcase}. Finally, in Section \ref{sectioncounterexamples} we will make some remarks and present some counterexamples related to extension problems for convex functions.

\section{$C^{\infty}$ convex extensions}\label{section proof cinfty}

In this section we will prove Theorem \ref{maintheorem Cinfty}. By using Whitney's extension theorem we may and do assume that $f\in C^{\infty}(\R^n)$, with $J_y^m f= P_y^m$ for all $m\in \N$ and all $y\in C$, and that $f$ satisfies condition $(CW^m)$ on $C$ for every $m\in\N$. We may also assume that $f$ has a compact support contained in $C+B(0,2)$.

\subsection{Idea of the proof.} Let us give a rough sketch of the proof so as to guide the reader through the inevitable technicalities. We warn the reader, however, that what we now say we are going to do is not exactly what we will actually do. Our proof could be rewritten to match this sketch exactly, but at the cost of adding further technicalities, which we do not feel would be pertinent. 

This proof has two main parts. In the first part we will estimate the possible lack of convexity of $f$ outside $C$: by using the conditions $(CW^m)$, a Whitney partition of unity, and some ideas from the proof of the Whitney extension theorem in the $C^\infty$ case, we will construct a function $\eta\in C^\infty(\R)$ such that $\eta\geq 0$, $\eta^{-1}(0)=(-\infty, 0]$, and $\min_{|v|=1}D^{2}f(x)(v^2)\geq -\eta\left( d(x,C)\right)$ for every $x\in\R^n$. 

In the second part of the proof we will compensate the lack of convexity of $f$ outside $C$ with the construction of a function $\psi\in C^{\infty}(\R^n)$ such that $\psi\geq 0$, $\psi^{-1}(0)=C$, and
$\min_{|v|=1}D^{2}\psi(x)(v^2)\geq 2\eta\left( d(x,C)\right)$.
Then, by setting $F:=f+\psi$ we will conclude the proof of Theorem \ref{maintheorem Cinfty}.

There are many ways to construct such a function $\psi$. The essential point is to write $C$ as an intersection of a family of half-spaces, and then to make a weighted sum, or an integral, of suitable convex functions composed with the linear forms that provide those half-spaces. If the sequence of linear forms is equi-distributed, in the weighted sum approach, or if one uses a measure equivalent to the standard measure on $\mathbb{S}^{n-1}$, in the integral approach, then the different functions $\psi$ produced by these methods will have equivalent convexity properties. See \cite{AFPAMS2002} for an instance of the weighted sum approach, and \cite[Proposition 2.1]{GhomiBLMS2004} for the integral approach. 

Of course our situation is more complicated than that of these references, as we need to find quantitative estimations of the convexity of $\psi$ outside $C$ which are strong enough to outweigh our previous estimations of the lack of convexity of $f$ outside $C$. It turns out that, in the present $C^{\infty}$ case, this goal can be achieved with either method of construction of $\psi$. Here we will follow the integral approach of Ghomi's in \cite[Proposition 2.1]{GhomiBLMS2004}, as it will lead us to easier calculations.

\subsection{First lower estimates for the Hessian of $f$: the numbers $\lbrace r_m \rbrace_m$}
We next show how the assumption of conditions $(CW^m)$ for every $m\geq 2$ implies a lower bound for the Hessian of $f$ in terms of the distance to $C.$

\begin{lemma}\label{inequalities rm}
For every $m\in \N,$ with $m \geq 2,$ there exists a number $r_m>0$ such that, whenever $d(x,C) \leq r_m,$ we have
$$
D^2f(x)(v^2) \geq - d(x,C)^{m-2} \quad \text{for every} \quad v\in \Sn.
$$
\end{lemma}
\begin{proof}
Fix $m\in \N,$ with $m \geq 2,$  $x\in\R^n\setminus C$ and $v\in \Sn.$ Let $y$ be the unique point of $C$ with the property that $d(x,C)=|x-y|.$ Set $t:=d(x,C)$ and $w= (x-y)/|x-y|.$ By Taylor's Theorem, we can write
\begin{align*}
D^2f(x)(v^2)  =& D^2f(y)(v^2)+ t\: D^3f(y)(w,v^2)+ \cdots + \frac{t^{m-2}}{(m-2)!} D^mf(y)(w^{m-2},v^2) \\
& + \frac{t^{m-2}}{(m-2)!} \left[ D^m f(y+sw)(w^{m-2},v^2)-D^mf(y)(w^{m-2},v^2) \right],
\end{align*}
for some $s\in [0,t].$ Since $f$ satisfies the condition $(CW^m),$ there exists a positive number $r_m$, independent of $y,v$ and $w$, such that
$$
\inf_{0<r\leq r_m } \left\lbrace \frac{D^2f(y)(v^2) + r D^3f(y)(w,v^2) + \cdots + \frac{r^{m-2}}{(m-2)!} D^mf(y)(w^{m-2},v^2)}{r^{m-2}} \right\rbrace \geq -\frac{1}{2}
$$
Thus, if $0<t\leq r_m,$ then
$$
D^2f(x)(v^2) \geq -\frac{t^{m-2}}{2} + \frac{t^{m-2}}{(m-2)!} \left[ D^m f(y+sw)(w^{m-2},v^2)-D^mf(y)(w^{m-2},v^2) \right].
$$
On the other hand, if $s\in [0,t],$ we can write
$$
D^m f(y+sw)(w^{m-2},v^2)-D^mf(y)(w^{m-2},v^2) \leq \| D^m f(x+sw)-D^mf(y) \| ;
$$
where we denote $\|A\|:= \sup_{u_i\in \Sn}|A(u_1, \ldots,u_m)|,$ for every $m$-linear form $A$ on $\R^n.$ Moreover, the above term is smaller than or equal to
$$
\varepsilon_m(t):= \sup_{\left\lbrace z\in\R^n, \, z'\in \partial C, \: |z-z'| \leq t \right\rbrace} \| D^mf(z)-D^mf(z')\|.
$$
Since $D^mf$ is uniformly continuous, there exists $r_m'>0$ such that, if $0<r\leq r_m',$ then $\varepsilon_m(r) \leq \frac{1}{2}$ (in fact we have $\lim_{r\to 0^{+}}\varepsilon_{m}(r)=0$). Therefore, assuming $0<t \leq \min \lbrace r_m,r_m'\rbrace ,$ we obtain
$$
D^2f(x)(v^2) \geq -\frac{t^{m-2}}{2}-\frac{t^{m-2}}{(m-2)!} \varepsilon_m (t) \geq -t^{m-2}.
$$
\end{proof}

\subsection{A Whitney partition of unity on $(0,+\infty)$}
For all $k\in \Z,$ we define the closed intervals
$$
I_k=[2^k , 2^{k+1}], \quad I_k^*=\left[ \frac{3}{4}2^k, \frac{9}{8}2^{k+1} \right].
$$
Obviously $(0,+ \infty)= \bigcup_{k \in \Z} I_k.$ Notice that $I_k$ and $I_k^*$ have the same midpoint and $\ell(I_k^*)= \frac{3}{2}\ell(I_k),$ where $\ell(I_k)=2^k$ denotes the length of $I_k.$ In other words, the interval $I_k^*$ is $I_k$ expanded by the factor $3/2.$
\begin{proposition}\label{properties intervals} The intervals $I_k$, $I_k^{*}$ satisfy the following.
\begin{enumerate}
\item[\bf{1.}] If $t\in I_k^*,$ then
$$
\frac{3}{4}\ell(I_k) \leq t \leq \frac{9}{4} \ell(I_k).
$$
\item[\bf{2.}] If $I_k^*$ and $I_j^*$ are not disjoint, then
$$
\frac{1}{2}\ell(I_k) \leq \ell(I_j) \leq 2 \ell(I_k).
$$
\item[\bf{3.}] Given any $t>0,$ there exists an open neighbourhood $U_t \subset (0,+\infty)$ of $t$ such that $U_t$ intersects at most $2$ intervals of the collection $\lbrace I_k^* \rbrace_{k\in \Z}.$
\end{enumerate}
\end{proposition}
This is a special case of the decomposition of an open set in Whitney's cubes, see \cite[Chapter VI]{Stein} for instance. In the one dimensional case things are much simpler and, for instance, it is easy to see that one may replace the number $N=12$ in \cite[Proposition VI.1.2, p. 169]{Stein} with the number 2. Anyhow, dealing with the number $12$ instead of $2$ would have no harmful effect in our proof.

In what follows we will not invoke the explicit definition of the intervals $I_k,$ $I_k^*$, as we will only need the properties mentioned in Proposition \ref{properties intervals}. {\bf Thus we may relabel the families $\lbrace I_k \rbrace_k$ and $\lbrace I_k^* \rbrace_k$, $k\in\Z$, as sequences indexed by $k\in\N$}, writing $\lbrace I_k \rbrace_{k\geq 1}$ and $\lbrace I_k^* \rbrace_{k \geq 1}.$ For every $k \geq 1,$ we will denote by $t_k$ and $\ell_k$ the midpoint and the length of $I_k$, respectively.

Next we recall how to define a Whitney partition of unity subordinated to the intervals $I_k^*$.
Let us take a bump function $\theta_0 \in C^\infty(\R)$ with $0 \leq \theta_0 \leq  1, \:\theta_0 = 1$ on $[-1/2,1/2];$ and $\theta_0 = 0$ on $\R\setminus (-\frac{3}{4},\frac{3}{4}).$ For every $k,$ we define the function $\theta_k$ by
$$
\theta_k(t)= \theta_0\left( \frac{t-t_k}{\ell_k} \right),  \quad t\in \R.
$$
It is clear that $\theta_k \in C^\infty(\R)$, that $0 \leq \theta_k \leq  1$, that $\theta_k = 1$ on $I_k$, and that $\theta_k = 0$ outside $\interior(I_k^*)$.

Now we consider the function $ \Phi = \sum_{k \geq 1} \theta_k$ defined on $(0,+\infty).$ Using Proposition \ref{properties intervals}, every point $t>0$ has an open neighbourhood which is contained in $(0,+\infty)$ and intersects at most two of the intervals $\lbrace I_k^{*} \rbrace_k.$ Since $\sop(\theta_k) \subset I_k^*,$ the sum defining $\Phi$ has only two terms and therefore $\Phi $ is of class $C^\infty.$ For the same reason, $\Phi (t)= \sum_{I_k^* \ni t} \theta_k(t) \leq 2,$ for $t>0.$ On the other hand, every $t>0$ must be contained in some $I_k,$ where the function $\theta_k$ takes the constant value $1,$ so we have $1 \leq \Phi \leq 2.$ These properties allow us to define, on $(0, +\infty),$ the functions $\theta_k^*= \frac{\theta_k}{\Phi}$. These are $C^\infty$ functions satisfying $\sum_k \theta_k^* =1$, $0\leq \theta_k^*\leq 1$, and $\sop(\theta_k^*) \subseteq I_k^*.$
Less elementary, but crucial, is the following property; see \cite{Whitney, Stein} for a proof.
\begin{proposition}\label{constants for derivatives}
For every $j \in \N \cup \lbrace 0 \rbrace,$ there is a positive constant $A_j$ such that
$$
|(\theta_k^*)^{(j)}(t) | \leq A_j \ell_k^{-j} \quad \text{for every} \quad t>0, \: k \in \N.
$$
\end{proposition}

\medskip

\subsection{The sequence $\lbrace \delta_p \rbrace_p$ and the function $\varepsilon$}\label{numbers delta and function epsilon}
Let us consider the numbers $r_m$ of Lemma \ref{inequalities rm}. We can easily construct a sequence $\lbrace \delta_p \rbrace_p$ of positive numbers satisfying
\begin{align*}
 \delta_p \leq &\min \left\lbrace r_{p+2} , \frac{1}{(p+2)!} \right\rbrace \quad \textrm{for } \:  p \geq 1, \\
& \delta_p < \frac{\delta_{p-1}}{2} \quad \textrm{for } \:  p \geq 2.
\end{align*}
Of course the sequence $\lbrace \delta_p \rbrace_p$ is decreasing to $0.$ Now, for every $k$ we define a positive integer $\gamma_k$ as follows. In the case that $\ell_k \geq \delta_1,$ we set $\gamma_k=1.$ In the opposite case, $\ell_k < \delta_1,$ we take $\gamma_k$ as the unique positive integer for which
$$
\delta_{\gamma_k+1} \leq \ell_k < \delta_{\gamma_k}.
$$
Finally let us define
$$
 \varepsilon(t)   :=  \left\lbrace
	\begin{array}{ccl}
	\sum_{k \geq 1} t^{\gamma_k} \theta_k^*(t) & \mbox{if } t>0, \\
	0  & \mbox{if }  t\leq 0.
	\end{array}
	\right.
$$
In the following lemma we show that $\varepsilon$ is of class $C^\infty$ on $\R$ and satisfies an additional technical property which will be important in Section \ref{subsectionconvexityneighbourhood}.  
\begin{lemma}\label{propertiesvarpesilon}
The function $\varepsilon$ satisfies the following properties.
\begin{itemize}
\item[$(1)$] $\varepsilon$ is of class $C^\infty(\R)$ and satisfies $\varepsilon^{(j)}(0)=0$ for every $j\in \N \cup \lbrace 0 \rbrace.$
\item[$(2)$] If $0<t\leq \delta_4 $ and $q\in \N$ are such that $\delta_{q+1} \leq t < \delta_q$ and $\frac{t}{2}\leq s \leq t,$ then $\varepsilon(2s) \geq t^{q+2}.$
\end{itemize}

\end{lemma}
\begin{proof} For the first statement, we immediately see that $\varepsilon^{-1}(0)=(-\infty , 0],$ that $\varepsilon>0$ on $(0,+\infty)$ and that $\varepsilon \in C^\infty(\R\setminus\lbrace 0 \rbrace ).$ In order to prove the differentiability of $\varepsilon$ at $t=0$ and that all the derivatives of $\varepsilon$ at $t=0$ are $0,$ it is sufficient to show that for all $j \in \N\cup\{0\},$
$$
\lim_{t \to 0^+} \frac{|\varepsilon^{(j)}(t)|}{t}=0.
$$
To check this, fix $j\in \N\cup\{0\}$ and $\eta>0$ and take
$$
\widetilde{t_j} := \min \left\lbrace \frac{\eta}{ 2B_j 4^j (j+1)! }, \delta_{j+5} \right\rbrace, \quad \text{where} \quad B_j= \max \lbrace A_l \: : \: 0 \leq l \leq j \rbrace.
$$
Recall that the numbers $A_l$ are those given by Proposition \ref{constants for derivatives}. Let $0<t \leq \widetilde{t_j}$. Since $\lbrace \delta_p \rbrace_p$ is decreasing, we can find a unique positive integer $q$ such that $\delta_{q+1} \leq t < \delta_q$. Because $t\leq \delta_{j+5} < \delta_1$, we must have $q \geq j+4.$ Now, if $k$ is such that $t\in I_k^*,$ Proposition \ref{properties intervals} tells us that
$$
\ell_k \leq \frac{4}{3}t < 2t \leq 2\delta_{j+5} < \delta_1,
$$ and using the definition of $\gamma_k,$ we have
$$
\delta_{\gamma_k +1} \leq \ell_k \leq \frac{4}{3}t < 2t <2\delta_q < \delta_{q-1}.
$$
The above inequalities imply that $\gamma_k+1 >q-1,$ that is $\gamma_k \geq q-1.$ In particular $\gamma_k \geq j+3.$ On the other hand, using Proposition \ref{properties intervals} again, we obtain
$$
\delta_{\gamma_k} > \ell_k \geq \frac{4t}{9}\geq \frac{t}{4} \geq \frac{\delta_{q+1}}{4} > \delta_{q+3},
$$
and hence $\gamma_k \leq q+2.$

If we use Leibniz's Rule, we obtain
$$
\varepsilon^{(j)}(t) = \sum_{k \geq 1} \sum_{l=0}^j \binom{j}{l} \frac{d^l}{dt^l}(t^{\gamma_k}) (\theta_k^*)^{(j-l)}(t),
$$
and since $\gamma_k \geq j+3$ if $t\in I_k^*,$ we can write
$$ \frac{|\varepsilon^{(j)}(t)|}{t} = \Bigg| \sum_{I_k^* \ni t} \sum_{l=0}^j \binom{j}{l} \frac{\gamma_k!}{(\gamma_k-l)!}t^{\gamma_k-l-1} (\theta_k^*)^{(j-l)}(t) \Bigg| \leq  \sum_{I_k^* \ni t} \sum_{l=0}^j j! \:  \gamma_k ! \: t^{\gamma_k-l-1} A_{j-l} \ell_k^{l-j}.
$$
Now, by Proposition \ref{properties intervals} we know that  $\ell_k \geq \frac{4}{9}t \geq \frac{1}{4}t.$ Moreover, because $\gamma_k \leq q+2, $ we have $\gamma_k! \leq (q+2)!$ and the last sum is smaller than or equal to
$$
\sum_{I_k^* \ni t} \sum_{l=0}^j j! \:  (q+2)! \: t^{\gamma_k-l-1} \: A_{j-l} \: \frac{t^{l-j}}{4^{l-j}}.
$$
Writing $t^{\gamma_k-l-1} = t^2 t^{\gamma_k-l-3} \leq t \: \delta_q \: t^{\gamma_k-l-3}$, this sum is smaller than or equal to
$$
\sum_{I_k^* \ni t} \sum_{l=0}^j j! \,  (q+2)! \, t \: \delta_q \, t^{\gamma_k-l-3}\, A_{j-l} \, \frac{t^{l-j}}{4^{l-j}} \leq 
\left( 4^j j! B_j \sum_{I_k^* \ni t} \sum_{l=0}^j (q+2)! \delta_q \, t^{\gamma_k-j-3} \right) t .
$$
Bearing in mind that $t \leq \delta_{j+5} < 1$ and $\gamma_k \geq j+3,$ we must have $t^{\gamma_k-j-3} \leq 1.$ By construction of the sequence $\lbrace \delta_p \rbrace_p$ we have that $(q+2)! \: \delta_q \leq 1$, and using that the sum $\sum_{I_k^* \ni t} $ has at most $2$ terms, we obtain
$$
\frac{|\varepsilon^{(j)}(t)|}{t} \leq 4^j (j+1) j!  \: 2B_j  t  \leq 4^j (j+1)! \: 2B_j  \widetilde{t_j} \leq \eta.
$$
This completes the proof of $(1).$

Let us now prove $(2).$ First of all, observe that $\delta_{q+1} \leq t \leq 2s \leq 2t < 2\delta_q < \delta_{q-1}$, and in particular $q\geq 3.$ Let us suppose that $2s \in I_k^*.$ Using Proposition \ref{properties intervals},
$$
\delta_{\gamma_k+1} \leq \ell_k \leq \frac{4}{3} (2s) <2(2s)<2 \delta_{q-1}< \delta_{q-2},
$$
that is $\gamma_k \geq q-2.$ If we use Proposition \ref{properties intervals} again,
$$
\delta_{\gamma_k} > \ell_k \geq \frac{4 (2s)}{9} \geq \frac{(2s)}{4} \geq \frac{\delta_{q+1}}{4} >\delta_{q+3},
$$
and then $\gamma_k\leq q+2.$

Finally, notice that $2s \leq 2t < \delta_{q-1}< \delta_1 <1$, and due to the fact that $\gamma_k\leq q+2$ if $2s\in I_k^*,$ we have that $(2s)^{q+2} \leq (2s)^{\gamma_k}.$ Therefore we obtain the desired inequality,
$$
t^{q+2} \leq (2s)^{q+2}  = \sum_{I_k^*\ni 2s} (2s)^{q+2} \theta_k^*(2s) \leq \sum_{I_k^*\ni 2s} (2s)^{\gamma_k} \theta_k^*(2s) =\varepsilon(2s).
$$
\end{proof}

\subsection{The function $\varphi$} \label{sectioncontructionvarphi} Next we will modify the function constructed by Ghomi in \cite[Proposition 2.1]{GhomiBLMS2004} so that we can obtain fine quantitative estimates for its Hessian which compensate the lack of convexity of $f$ off of $C$ (see Lemma \ref{convexity in an ball} below). We begin by defining
$$
 \tilde{\varepsilon}(t)   =  \left\lbrace
	\begin{array}{ccl}
	\dfrac{\varepsilon(2t)}{t^{n+3}} & \mbox{ if }  t>0 \\
	0  & \mbox{ if }  t\leq 0.
	\end{array}
	\right.
$$
Since $\varepsilon\in C^\infty(\R)$, with $\varepsilon^{(j)}(0)=0$ for all $j\in\N\cup\{0\}$, we have that $ \tilde{\varepsilon}\in C^\infty(\R)$ and $\tilde\varepsilon^{(j)}(0)=0$ for all $j\in\N\cup\{0\}$ as well. Now, let us consider the function
$$
g(t)   =  \left\lbrace
	\begin{array}{ccl}
	\int_0^t \int_0^s  \tilde{\varepsilon}(r)dr \: ds & \mbox{ if }  t>0 \\
	0  & \mbox{if }  t\leq 0.
	\end{array}
	\right.
$$
It is clear that $g\in C^\infty(\R)$ and $g^{(j)}(0)=0$ for all $j\in \N \cup \{ 0 \}.$ In addition, $g^{-1}(0)=(-\infty ,0]$ and $g''(t)= \tilde{\varepsilon}(t)>0$ for all $t>0.$ In particular, $g$ is convex on $\R$ and positive, with a strictly positive second derivative, on $(0,+\infty).$ We may assume that $0\in C$. 

Now, for every vector $w\in \Sn,$ define $h(w)= \max_{z\in C} \langle z,w \rangle$, the support function of $C$ (for information about support functions of convex sets, see \cite{Rockafellar} for instance). We also define the function
 $$
\begin{array}{rccl}
\phi : &\Sn \times \Rn & \longrightarrow & \R  \\
   &(w, x) & \longmapsto & \phi(w,x) = g(\langle x, w \rangle -h(w)).
\end{array}
$$
It is easy to see that, for every $w=(w_1,\ldots,w_n) \in \Sn$ and every multi-index $\alpha,$ we have
$$
\frac{\partial^\alpha}{\partial x^\alpha} \phi ( w, x) = g^{(|\alpha|)}(\langle x, w \rangle -h(w)) w^\alpha, \quad \text{where} \quad w^\alpha=w_1^{\alpha_1} \cdots w_n^{\alpha_n}.
$$
Also observe that $\langle x ,w \rangle \leq h(w)$ for every $x\in C,$ $w\in \Sn.$ Therefore, the properties of $g$ and its derivatives imply that $\phi(w, \cdot)$ is a function of class $C^\infty(\Rn)$ whose derivatives of every order and itself vanish on $C,$ for every $w\in \Sn.$ It is also easy to check that the function $\phi(w,\cdot)$, being a composition of a convex function with a non-decreasing convex function, is convex as well.

Finally, we define the function $\varphi : \Rn \to \R$ as follows:
$$
\varphi(x) = \int_{\Sn} \phi(w,x) \: dw \quad \text{for every} \quad x\in \Rn.
$$
Again it is easy to check that $\varphi^{-1}(0)=C$ and $\varphi$ is convex. Because $\phi(w,\cdot)$ is of class $C^\infty(\Rn)$, the derivatives $(w,x)\mapsto \frac{\partial^{\alpha}}{\partial x^{\alpha}}\phi(w,x)$ are continuous for every multi-index $\alpha$, and because $\mathbb{S}^{n-1}$ is compact, it follows from standard results on differentiation under the integral sign that the function $\varphi$ is of class $C^\infty(\R^n)$ as well and that $\partial^\alpha \varphi (y)=0$ for every $y\in C$ and every multi-index $\alpha.$ In other words, $J_y ^m \varphi =0$ for all $m\in \N \cup \lbrace 0 \rbrace$ and all $y\in C.$ One can also check easily that
$$
D^2\varphi(x)(v^2) = \int_{\Sn} g'' ( \langle x,w \rangle-h(w) ) \langle w,v \rangle^2 \: dw  \quad \text{for every} \quad x\in  \R^n, \: v\in \Sn.
$$

\subsection{Selection of angles and directions}\label{sectionangles} Given $x\in\R^n\setminus C$ and $v\in\mathbb{S}^{n-1}$ we will now find a region $W=W(x,v)$ of $\mathbb{S}^{n-1}$ of sufficient volume (depending only, and conveniently, on $d(x,C)$) on which we have good lower estimates for $g'' ( \langle x,w \rangle-h(w) ) \langle w,v \rangle^2$. This will involve a careful selection of angles and directions.

Fix a point $x\in\R^n\setminus C$, let $x_C$ be the metric projection of $x$ onto the compact convex $C$, and define
\begin{equation}\label{definitionsprojectionnormalangle}
u_x:=\frac{1}{|x-x_C|}(x-x_C)\quad \text{and} \quad \alpha_x:= \frac{d(x,C)}{d(x,C)+ \diam(C)}.
\end{equation}
\begin{lemma}\label{choice of alpha}
We have that $d(x,C)=\langle x,u_x \rangle - h(u_x) $ and
$$
d(x, C)\geq \langle x, w \rangle -h(w) \geq \frac{1}{2}d(x,C)
$$
for all $w\in \Sn$ such that $\widehat{ w \: u_x} \in \left[ \frac{\alpha_x}{3}, \frac{\alpha_x}{2} \right]$.
\end{lemma}
\noindent Here $\widehat{ w \: u_x}$ denotes the length of the shortest geodesic (or angle) between $w$ and $u_x$ in $\mathbb{S}^{n-1}$.
\begin{proof} The fact that $\langle x,u_x \rangle - h(u_x)  = d(x,C)$ is a straightforward consequence of the definition of $h$ and $u_x$. For the second part,
given $w\in \Sn$ with $\widehat{ w \: u_x} \in \left[ \frac{\alpha_x}{3}, \frac{\alpha_x}{2} \right]$, let us denote $\theta=\widehat{ w \: u_x}$. Since $C$ is compact, we can find $z \in C$ such that $h(w)=\langle z,w\rangle.$ Using that $\langle x,u_x \rangle - h(u_x) =|x-x_C|$ and $|w-u_x|\leq \theta$, we have
\begin{align*}
 \langle x, w \rangle -h(w) & = \langle x,w-u_x \rangle+ |x-x_C|+h(u_x)-h(w) \\
& \geq  \langle x,w-u_x \rangle+ |x-x_C|+\langle z, u_x-w\rangle\\
&  =  \langle x-z, w-u_x\rangle +|x-x_C| \\
& \geq -\left(\textrm{diam}(C)+|x-x_C|\right)\theta+|x-x_C| \\
& \geq -\left(\textrm{diam}(C)+|x-x_C|\right)\frac{\alpha_x}{2}+|x-x_C| \\
& = \frac{1}{2}|x-x_C|.
\end{align*}
The other inequality, $d(x, C)\geq \langle x, w \rangle -h(w)$, follows immediately from the definition of $h.$ 
\end{proof}
Next we find the region $W$ we need.
\begin{lemma} \label{choice of w}
Given any $v\in \Sn$ with $\langle u_x, v \rangle \geq 0,$ there exists a vector $w_0=w_0(x,v) \in \Sn$ such that the set
$$
W:= \lbrace w\in \Sn \: : \: \widehat{ w \: w_0} \in [0,\tfrac{\alpha_x}{12} ] \rbrace
$$
satisfies the following.
\begin{enumerate}
\item[$(1)$] For every $w\in W$ we have $\widehat{ u_x \: w} \in \left[ \frac{\alpha_x}{3}, \frac{\alpha_x}{2} \right].$
\item[$(2)$] For every $w\in W$ we have $\langle w, v \rangle \geq \sin(\frac{\alpha_x}{3}).$
\item[$(3)$] $\vol_{\Sn}(W) \geq V(n) \alpha_x^{n-1}$, where $V(n)>0$ is a constant depending only on the dimension $n.$
\end{enumerate}
\end{lemma}
\begin{proof} We prove $(1)$ and $(2)$ at the same time by studying two cases separately.

\noindent \textbf{Case 1.} $u_x \neq v.$ Take an $w_0$ in the unit circle of the plane spanned by the vectors $u_x$ and $v$, in such a way that $\widehat{w_0 \: u_x}=\frac{5\alpha_x}{12}$, and that the arc in that circle joining $u_x$ with $w_0$ has the same orientation as the arc joining $u_x$ with $v.$ Set $W= \lbrace w\in \mathbb{S}^{n-1} \: : \: \widehat{w\:  w_0} \in [0,\frac{\alpha_x}{12} ] \rbrace$ and let $w\in W$. Bearing in mind that the angles shorter than $\pi$ give the usual intrinsic distance in $\mathbb{S}^{n-1}$, we may use the triangle inequality for this distance in order to estimate
$$
\widehat{u_x\: w} \leq \widehat{u_x\: w_0}+ \widehat{w_0\:w} \leq \frac{5\alpha_x}{12} + \frac{\alpha_x}{12}= \frac{\alpha_x}{2},
$$
$$
\widehat{u_x\: w} \geq \widehat{u_x\:  w_0}- \widehat{w_0 \:w} \geq \frac{5\alpha_x}{12} - \frac{\alpha_x}{12}= \frac{\alpha_x}{3},
$$
that is, $\widehat{u_x\: w} \in \left[ \frac{\alpha_x}{3},\frac{\alpha_x}{2} \right].$ It only remains to see that $\langle w, v \rangle \geq \sin ( \frac{\alpha_x}{3} )$ for all $w\in W.$ It is easy to verify $ \widehat{v \: w_0} \leq \frac{\pi}{2}-\frac{5\alpha_x}{12}$ and, then, for every $w\in W,$ we have
$$
\widehat{w \: v} \leq \widehat{w \: w_0}+ \widehat{w_0 \: v} \leq \frac{\alpha_x}{12} + \frac{\pi}{2}- \frac{5\alpha_x}{12} = \frac{\pi}{2}- \frac{\alpha_x}{3}.
$$
Therefore $\langle w,v \rangle = \cos( \widehat{w\: v} ) \geq \cos( \frac{\pi}{2}-\frac{\alpha_x}{3} ) = \sin ( \frac{\alpha_x}{3} ).$

\noindent \textbf{Case 2.} $ u_x=v.$ Pick $w_0 \in \mathbb{S}^{n-1}$ such that $\widehat{w_0 \: u_x}=\frac{5\alpha_x}{12}.$ With the same estimations as in Case 1 we obtain $\widehat{u_x \:  w} \in [\frac{\alpha_x}{3},\frac{\alpha_x}{2}],$ and hence $\langle w , v \rangle = \langle w , u_x \rangle \geq \cos \left( \frac{\alpha_x}{2} \right) \geq \sin \left( \frac{\alpha_x}{3} \right)$ for every $w\in W.$ 

\medskip

\noindent Let us now prove $(3).$  The set $W$ is a hyperspherical cap, and its volume is given by
$$
 \vol_{\Sn}(W)  = \vol(\mathbb{S}^{n-2}) \int_0^{\alpha_x/12} \sin^{n-2}(\beta) d\beta,
$$
where $\vol(\mathbb{S}^{n-2})=2$ in the special case $n=2.$ For those angles $\beta$ such that $0 \leq \beta \leq \frac{\alpha_x}{12} \leq \frac{\pi}{3},$ it is clear that $\sin\beta \geq \frac{1}{2} \beta,$ and therefore
$$
 \vol_{\Sn}(W) \geq \vol(\mathbb{S}^{n-2})  \int_0^{\alpha_x/12} (\tfrac{1}{2} \beta)^{n-2} d\beta = V(n) \alpha_x^{n-1}; \quad \text{where}
$$
$$
 V(n)=\dfrac{\vol(\mathbb{S}^{n-2})}{12(n-1) (24)^{n-2}} \quad \text{for every} \quad n \geq 2. 
$$ 
\end{proof}

\subsection{Convexity of $f+\psi$ on a neighbourhood of $C$}\label{subsectionconvexityneighbourhood}
Now, using the constant $V(n) $ obtained in Lemma \ref{choice of w} $(3)$, define $$C(n):= \frac{V(n)}{36(1+\diam(C))^{n+1}}.$$
\begin{lemma} \label{convexity in an ball}
With the notation of Section \ref{numbers delta and function epsilon}, let us define $r:=\delta_4$ and consider the function $H=f+\frac{2}{C(n)} \varphi$ defined on $\R^n.$ Then, for every $x \in\R^n\setminus C$ such that $t:=d(x,C) \leq r$, and for every $v\in \Sn,$ we have
$$
D^2H(x)(v^2) \geq  t^q,
$$
where $q$ is the unique positive integer such that $\delta_{q+1} \leq t < \delta_q.$
\end{lemma}
\begin{proof}
Fix $x, t, v, q$ as in the statement. Since $D^2H(x)(v^2)= D^2H(x)((-v)^2)$, we may suppose that $ \langle v, u_x \rangle \geq 0,$ where $u_x$ is that of \eqref{definitionsprojectionnormalangle}. Consider the set $W=W(x,v)$ as in Lemma \ref{choice of w}. By the construction of $\varphi,$ we have
\begin{equation}\label{second differential of varphi}
  D^2\varphi(x)(v^2)= \int_{\Sn} \tilde{\varepsilon} ( \langle x,w \rangle-h(w) ) \langle w,v \rangle^2 \: dw
 \geq \int_{W} \tilde{\varepsilon} ( \langle x,w \rangle-h(w) ) \langle w,v \rangle^2 \: dw >0.
\end{equation}
For every $w\in W,$ Lemma \ref{choice of w} $(1)$ says that $ \widehat{ w \: u_x} \in \left[ \frac{\alpha_x}{3}, \frac{\alpha_x}{2} \right];$ where $\alpha_x$ is that of \eqref{definitionsprojectionnormalangle}. According to Lemma \ref{choice of alpha}, we have
$$
\frac{t}{2} \leq \langle x,w \rangle-h(w) \leq t \leq \delta_4.
$$
Thus we obtain from Lemma \ref{propertiesvarpesilon} $(2)$ that
\begin{equation}\label{inequalitydistanceq-n-1}
\tilde{\varepsilon} ( \langle x,w \rangle-h(w) ) = \dfrac{\varepsilon(2(\langle x,w \rangle-h(w)))}{(\langle x,w \rangle-h(w))^{n+3}} \geq \dfrac{t^{q+2}}{(\langle x,w \rangle-h(w))^{n+3}} \geq \dfrac{t^{q+2}}{t^{n+3}} =  \dfrac{t^{q}}{t^{n+1}}.
\end{equation}
On the other hand, $\langle w,v \rangle \geq \sin(\frac{\alpha_x}{3})$ for every $w\in W,$ by virtue of Lemma \ref{choice of w} $(2)$. By combining the inequalities \eqref{second differential of varphi} and \eqref{inequalitydistanceq-n-1}, we get
\begin{equation}\label{estimationhessianwithW}
D^2\varphi(x)(v^2) \geq \dfrac{t^{q}}{t^{n+1}} \sin^2 \left( \frac{\alpha_x}{3} \right) \vol_{\Sn}(W).
\end{equation}
Since $\alpha_x \leq 1,$ we obviously we have $\sin(\frac{\alpha_x}{3} )\geq \frac{\alpha_x}{6}.$ Also, bearing in mind that $t\leq r = \delta_4 <1,$ it is clear that
$$
\alpha_x= \frac{t}{t+\diam(C)} \geq \frac{t}{1+\diam(C)}.
$$
In addition, Lemma \ref{choice of w} $(3)$ gives us $ \vol_{\Sn}(W) \geq  V(n) \alpha^{n-1}_x.$ Combining these inequalities with \eqref{estimationhessianwithW} we obtain
\begin{equation}\label{estimationhessianwithCn}
D^2\varphi(x)(v^2) \geq \dfrac{t^{q}}{t^{n+1}} \frac{t^{n+1}}{36 (1+ \diam(C))^{n+1} } V(n) = C(n) t^q .
\end{equation}
Finally, by the construction of the sequence $\lbrace \delta_p \rbrace_p,$ (see Section \ref{numbers delta and function epsilon}) we have $d(x,C) = t < \delta_q \leq r_{q+2}.$ According to Lemma \ref{inequalities rm}, we obtain the inequality
$$
D^2f(x)(v^2) \geq - t^q;
$$
which, in combination with \eqref{estimationhessianwithCn}, yields
$$
D^2H(x)(v^2) = D^2f(x)(v^2)+ \frac{2}{C(n)} D^2\varphi(x)(v^2) \geq -t^q + 2t^q= t^q.
$$
\end{proof}
Since $J_y^m \varphi=0$ for $y\in C$ and each $m\in \N \cup \lbrace 0 \rbrace, $ we have proved that $H$ is of class $C^\infty ( \Rn)$, $H=f$ on $C$, $J_y^m H = J_y^m f=P_y^m$ for every $y\in C$ and every $m\in \N$, and $H$ has a strictly positive Hessian on the set $\lbrace x\in \R^n : 0<d(x,C)\leq r \rbrace.$

\subsection{Conclusion of the proof: convexity of $f+\psi$ on $\R^n$}\label{conclusionproofcinfty}
In order to complete the proof of Theorem \ref{maintheorem Cinfty} we only have to change the function $H$ slightly.
\begin{lemma}
There exists a number $a>0$ such that the function $F:=f+a\varphi$ is of class $C^\infty(\R^n)$, concides with $f$ on $C$, satisfies $J_y^m F =P_y^m$ for every $y\in C$, $m\in \N$, is convex on $\R^n$, and has a strictly positive Hessian on $\R^n \setminus C$.
\end{lemma}
\begin{proof}
Let us denote $\psi=\frac{2}{C(n)}\varphi$.
We recall that $f=0$ outside $C+B(0,2)$. Take $r>0$ as in Lemma \ref{convexity in an ball}. Since $C_r:=\lbrace x \in \R^n \: : \: r\leq d(x,C) \leq 2 \rbrace$ is a compact subset where $\psi$ has a strictly positive Hessian (cf. (\ref{second differential of varphi})), and using again that $f$ has compact support, we can find $ M \geq 1$ such that
$$
\sup_{x\in\R^n, \, v\in \Sn} | D^2f(x)(v^2) | \leq M \quad \text{ and } \quad
\underset{ x\in C_r, \: v\in \Sn } {\inf} D^2\psi(x)(v^2) \geq \frac{1}{M}.
$$
Let us take $A=2M^2$ and $F=f+A\psi.$ If $d(x,C) \leq r$ and $v\in\mathbb{S}^{n-1}$ we have, by Lemma \ref{convexity in an ball}, that
$$
D^2F(x)(v^2) = 2M^2 D^2\psi(x)(v^2) + D^2f(x)(v^2) > D^2\psi(x)(v^2)+D^2f(x)(v^2) >0.
$$
In the case when $d(x,C) \in [r,2],$ given any $|v|=1,$ we easily see that
$$
D^2F(x)(v^2) = 2M^2 D^2\psi(x)(v^2) + D^2f(x)(v^2) \geq 2M -M =M>0.
$$
Finally, in the region $\lbrace x \: : \: d(x,C) > 2 \rbrace,$ we have that $f=0.$ Hence
$$
D^2F(x)(v^2) = 2M^2 D^2\psi(x)(v^2) >0.
$$
Therefore, in any case, by setting $a=2A/C(n)$, we get that $F = f+ A\psi=f+a\varphi$ is of class $C^\infty(\Rn)$, satisfies $F(y)=f(y)$ and $J_y^m F=P_y^m$ for every $y\in C$, $m\in \N$, and has a positive Hessian on $\R^n \setminus C.$ Since $f$ is convex on $C$ and $F$ is differentiable, this is easily seen to imply that $F$ is convex on all of $\R^n$.
\end{proof}

\subsection{Proof of Corollary \ref{corollary for Cm and CWk with a strict inequality}}\label{proof of corollary in case m finite}
An obvious variation of the proof of the above lemma shows Proposition \ref{if f is convex on a neighbourhood of C then the problem is easy}. On the other hand Proposition \ref{if f is convex on a neighbourhood of C then the problem is easy} can easily be used to show Corollary \ref{corollary for Cm and CWk with a strict inequality}. Indeed, we have
$$
D^2f(y)(v^2)+ t\: D^3f(y)(w,v^2)+ \cdots + \frac{t^{k-2}}{(k-2)!} D^kf(y)(w^{k-2},v^2)\geq \eta t^{k-2}
$$
for all $y\in C$, $w,v\in\mathbb{S}^{n-1}$, $0<t\leq t_0$ and, also, by Taylor's theorem and uniform continuity of $D^{m}f$,
\begin{eqnarray*}
& & D^2f(y+tw)(v^2) =\\
 & & D^2f(y)(v^2)+ t\: D^3f(y)(w,v^2)+ \cdots + \frac{t^{m-2}}{(m-2)!} D^mf(y)(w^{m-2},v^2) +R_{m}(t, y, v, w),
\end{eqnarray*}
where
$$
\lim_{t\to 0^{+}}\frac{R_{m}(t,y,v,w)}{t^{m-2}}=0 \quad \textrm{uniformly on} \quad y\in C, w, v\in \mathbb{S}^{n-1}.
$$
We may assume $t_0\leq 1$.
Then we may also find $t_0'\in (0, t_0)$ such that $R_{m}(t, y, v, w)\geq -\frac{\eta}{2}t^{m-2}$ for all $y\in C$, $w,v\in\mathbb{S}^{n-1}$, $0<t\leq t'_0$, and it follows that
$$
D^2f(y+tw)(v^2) \geq \frac{\eta}{2} t^{m-2}
$$
for all $y\in C$, $w,v\in\mathbb{S}^{n-1}$, $0<t\leq t'_0$. This implies that $D^{2}f(x)\geq 0$ whenever $d(x,C)\leq t'_0$, and therefore that $f$ is convex on $U:=\{x\in\R^n \, : \, d(x,C)<t'_0\}$. Corollary \ref{corollary for Cm and CWk with a strict inequality} then follows from Proposition \ref{if f is convex on a neighbourhood of C then the problem is easy}.
\bigskip

\section{$C^{m}$ convex extensions for $m\geq 2$ finite}\label{sectionproofcm}

We start this section with the proof of Theorem \ref{maintheoremfinitecase}. We may assume that $C$ is a compact convex subset of $\Rn$ and $f:\Rn \to \R$ is of class $C^m(\Rn),$ $m \geq n+3,$ with support contained on $C+B(0,2)$, and such that $f$ satisfies condition $(CW^m)$ on $C$. We will split the proof into several subsections.

\subsection{The function $\omega$} 
\begin{lemma} \label{functionomega}
There exists a non decreasing continuous function $\omega: [0, +\infty) \to [0,+\infty)$ with $\omega(0)=0$ and such that
$$
D^2f(x)(v^2) \geq -\omega(d(x,C)) d(x,C)^{m-2} \quad \text{for all} \quad x\in \Rn,\: v\in \Sn.
$$
\end{lemma}
\begin{proof}
Let us define
$$
Q_m(t,y,v,w):= \frac{D^2f(y)(v^2)+ \cdots+ \frac{t^{m-2}}{(m-2)!} D^mf(y)(w^{m-2},v^2)}{t^{m-2}},
$$ 
$$
\varepsilon_m(t):= \sup_{\left\lbrace z\in\R^n, \, z'\in \partial C, \: |z-z'| \leq t \right\rbrace} \| D^mf(z)-D^mf(z')\|,$$for all $t>0, y\in C, \: v,w\in \Sn.$ Condition $(CW^m)$ together with the uniform continuity of $D^mf$ imply that, for every positive integer $p,$ there exists $r_p>0$ such that
\begin{equation}\label{boundforsmalldistances}
Q_m(t,y,v,w) \geq -\frac{1}{2p} \quad \text{and} \quad \varepsilon_m(t) \leq \frac{1}{2p} 
\end{equation}
for every $y\in \partial C,\:  v,w\in \Sn$ and $ 0<t \leq r_p.$ We may suppose that the sequence $\lbrace r_p \rbrace_{p \geq 1}$ is decreasing to $0.$ Since the derivatives of $f$ up to order $m$ are bounded on $\Rn$ we can find a constant $M>1$ such that
\begin{equation}\label{boundforlargedistances}
\varepsilon_m(t) - Q_m(t,y,v,w) \leq M \quad \text{for all} \quad y\in \partial C,\: v,w \in \Sn, \: t \geq r_1.
\end{equation}

Now, given $x\in \Rn \setminus C$ and $v\in \Sn,$ we denote by $y$ the metric projection of $x$ onto $C,\:  w:=(x-y)/|x-y|$ and $t:=d(x,C).$ By Taylor's theorem and the definition of $Q_m$ and $\varepsilon_m,$ we have
$$
D^2f(x)(v^2) \geq  t^{m-2} Q_m(t,y,v,w) - t^{m-2} \varepsilon_m(t)=-t^{m-2} \left( \varepsilon_m(t)- Q_m(t,y,v,w) \right).
$$
We define $\omega : [0,+\infty) \to [0,+\infty)$ by setting
\begin{align*}
& \omega(0)=0, \quad \omega(r_p)= \frac{1}{p-1} \quad  p \geq 2, \quad \omega(r_1)=M, 
\\
&  \omega \: \text{ affine on each } \:\: [r_{p+1}, r_p] \quad p \geq 1, \quad \omega(t)=M \quad t \geq r_1.
\end{align*}
It is easy to check that $\omega$ is a non-decreasing continuous function such that $\omega (t) \geq \frac{1}{p}$ for every $t \geq r_{p+1}$ and every $p \geq 2$, and that $\omega(t) \geq 1$ for every $t \geq r_2.$ Using inequalities \eqref{boundforsmalldistances} and \eqref{boundforlargedistances} we deduce that
\begin{align*}
& D^2f(x)(v^2) \geq -Mt^{m-2} \quad \text{for} \quad t \geq r_1 \\
& D^2f(x)(v^2) \geq -\frac{1}{p}t^{m-2} \quad \text{for} \quad t \leq r_p, \quad p \in \N.
\end{align*} 
By the properties of $\omega$ we conclude
$$
D^2f(x)(v^2) \geq -\omega(d(x,C)) d(x,C)^{m-2} \quad \text{for every} \quad x\in \Rn, \: v\in \Sn.
$$

\end{proof}

\subsection{The function $\varphi$}\label{subsectionproofcm1}
Using the function $\omega$ of Lemma \ref{functionomega}, we introduce two new functions
$$
 g(t)   =  \left\lbrace 
	\begin{array}{ccl}
	\int_0^t \int_0^{t_2} \cdots \int_0^{t_{m-n-1}} \omega(2^{m-n-2}s)  ds \: dt_{m-n-1}\cdots dt_2 & \mbox{if } \: t>0 \\
	0  & \mbox{if } \: t\leq 0,
	\end{array}
	\right.
$$
$$
\varphi(x)= \int_{\Sn} g( \langle x,w \rangle-h(w) ) \: dw , \quad x\in \Rn.
$$
Since $\omega$ is continuous, the function $g$ is of class $C^{m-n-1}(\R)$ with $g^{(k)}(0)=0$ for every $1 \leq k \leq m-n-1.$ The same arguments and calculations as in Section \ref{sectioncontructionvarphi} allow us to deduce that $\varphi$ is of class $C^{m-n-1}(\Rn)$ with $\varphi^{-1}(0)=C$ and $J^{m-n-1}_y \varphi=0$ for all $y\in C.$ It is also easy to see
\begin{equation}\label{firstcmestimationhessian}
D^2\varphi (x)(v^2)  = \int_{\Sn} g'' (  \langle x,w \rangle-h(w) ) \langle v,w \rangle^2 dw \quad \text{for all} \quad x\in \Rn, \: v\in \Sn.
\end{equation}

\subsection{Conclusion of the proof of Theorem \ref{maintheoremfinitecase}}\label{subsectionproofcm2}
Let $x\in \Rn \setminus C$ be such that $t:=d(x,C) \leq 1$ and let $v\in \Sn.$ With the same calculations as in the proof of Lemma \ref{convexity in an ball} and bearing in mind that $g''$ is nondecreasing, we obtain from \eqref{firstcmestimationhessian} that
\begin{equation}\label{lastboundhessian}
D^2\varphi(x)(v^2) \geq  \frac{V(n)}{36} \left(\frac{t}{1+\diam(C)} \right)^{n+1}  g''  \left( \frac{t}{2} \right),
\end{equation}
where $V(n)>0$ is the constant given by Lemma \ref{choice of w} $(3).$ Let us now estimate $ g'' (t/2).$ By the construction of $g$ we have
$$
 g'' \left( \frac{t}{2} \right) = \int_0^{t/2} \int_0^{t_2} \cdots \int_0^{t_{m-n-3}} \omega (2^{m-n-2}s)  ds \: dt_{m-n-3}\cdots dt_2,
$$
where, in the special case $m=n+3,$ the above expression means $g'' (t/2) = \omega (t).$ Using that $\omega$ is nonnegative and nondecreasing we can write
\begin{align*}
g'' \left( \frac{t}{2} \right) &\geq  \int_{t/4}^{t/2} \int_0^{t_2} \cdots \int_0^{t_{m-n-3}} \omega (2^{m-n-2}s)  ds \: dt_{m-n-3}\cdots dt_2 \\
& \geq \frac{t}{4} \int_0^{t/4}  \int_0^{t_2} \cdots \int_0^{t_{m-n-4}} \omega (2^{m-n-2}s)  ds \: dt_{m-n-4}\cdots dt_2 \\
& \geq \frac{t}{4}\cdot\frac{t}{8} \int_0^{t/8}  \int_0^{t_2} \cdots \int_0^{t_{m-n-5}} \omega (2^{m-n-2}s)  ds \: dt_{m-n-5}\cdots dt_2 \\
& \geq \frac{t}{4} \cdot \frac{t}{8} \cdots \frac{t}{2^{m-n-3}} \cdot \frac{t}{2^{m-n-2}} \omega (t)= \frac{t^{m-n-3}}{2^{2+3+\cdots+(m-n-2)}} \omega (t).
\end{align*}
By plugging this estimation in \eqref{lastboundhessian}, we obtain
$$
D^2 \varphi(x)(v^2) \geq k(n,m,C) t^{m-2} \omega (t), \quad \text{where}
$$
$$
k(n,m,C):= \frac{V(n)}{36 \cdot 2^{2+3+\cdots+(m-n-2)} (1+\diam(C))^{n+1} }.
$$
On the other hand, Lemma \ref{functionomega} implies that
$$
D^2f(x)(v^2) \geq -\omega(t)t^{m-2}.
$$
Therefore, the function $\psi= f+ \frac{2}{k(n,m,C)} \varphi$ satisfies $D^2 \psi(x)(v^2) \geq 0$ on the neighbourhood $\lbrace x\in \Rn \: : \: d(x,C) \leq 1 \rbrace$ of $C$, with a strict inequality whenever $ 0<d(x,C) \leq 1.$ We also have that the function $\psi$ is of class $C^{m-n-1}(\Rn)$, with $f=\psi$ on $C$ and $J_y^{m-n-1} \psi = J_y^{m-n-1} f$ for all $y\in C.$ Finally, using the same argument as in Section \ref{conclusionproofcinfty}, we can construct a convex function $F\in C^{m-n-1}(\Rn)$ with $F=f$ on $C$ and $J^{m-n-1}_y F = J_y^{m-n-1}f$ for all $y\in C.$ The proof of Theorem \ref{maintheoremfinitecase} is now complete. \qed

\medskip

In the rest of the section we will give the proof of Theorem \ref{theoremspecialcase}.

\subsection{Sublevel sets of strongly convex functions} \label{subsectionsublevelsets}
Here we gather some elementary properties of ovaloids that we will need in the proof of Theorem \ref{theoremspecialcase}.
\begin{proposition} \label{stronglyconvexfunction}
Suppose that $\psi : \Rn \to \R$ is a convex function of class $C^m(\Rn),$ with $m\geq 2,$ such that there exists a constant $M>0$ with $D^2\psi(x)(v^2) \geq M$ for all $x\in\Rn$ and for all $v\in \Sn.$ If we denote $C=\psi^{-1}(-\infty,1],$ then the following is true.
\begin{enumerate}
\item[$(1)$] $C$ is a convex compact set, $\partial C= \lbrace x\in \Rn \: : \: \psi(x)=1 \rbrace$ and $\interior(C)=\lbrace x\in \Rn \: : \: \psi(x)<1 \rbrace.$
\item[$(2)$] If $\interior(C)=\emptyset,$ then $C$ is a singleton.
\end{enumerate}
If we further assume that $\interior(C) \neq \emptyset$ then we also have: 
\begin{enumerate}
\item[$(3)$] $\nabla \psi$ does not vanish on $\partial C$ and $\partial C$ is a one-codimensional manifold of class $C^m.$ 
\item[$(4)$] If $x\notin C$ and $x_C \in \partial C$ is such that $|x-x_C| = d(x,C),$ then $\nabla \psi(x_C)$ and $x-x_C$ are paralell and outwardly normal to $\partial C$ at the point $x_C$. 
\item[$(5)$] There is a constant $\beta >0$ such that
$$
\psi(x)-1 \geq \beta d(x,C) \quad \text{for every} \quad x\in \Rn \setminus C.
$$
\end{enumerate}
\end{proposition}
\begin{proof}
Properties $(1)-(4)$ are well-known facts about strongly convex functions of class $C^m$. Perhaps only property $(5)$ requires an explanation. The compactness of $\partial C$ together with $(3)$ implies that $\inf_{\partial C} |\nabla \psi| \geq \beta$ for some $\beta >0.$ If $x\notin C,$ and $x_C \in \partial C$ is such that $|x-x_C| = d(x,C),$ the convexity of $\psi$ yields
$$
\psi(x)-1= \psi(x)-\psi(x_C) \geq \langle \nabla \psi(x_C),x-x_C \rangle.
$$
According to $(4)$, the last term coincides with $ |\nabla \psi(x_C)| |x-x_C| \geq \beta d(x,C).
$
\end{proof}

By using well-known properties of the Minkowski functional of a convex set, one can easily show the following.
\begin{proposition} \label{distanceintersection}
If $C= \bigcap_{k=1}^NC_k,$ where each $C_k$ is a convex and bounded subset of $\Rn$ with $\interior(C) \neq \emptyset$ we have
$$
 \max_{1\leq k \leq N} d(x,C_k) \leq d(x,C) \leq \frac{R}{r} \max_{1\leq k \leq N} d(x,C_k) \quad \text{for all} \quad x\in \Rn,
$$
where $r,R >0$ are such that $\overline{B}(x_0,r) \subseteq C \subseteq \overline{B}(x_0,R)$ and $x_0 \in \interior(C).$
\end{proposition}

\medskip

\subsection{Proof of Theorem \ref{theoremspecialcase}}\label{conclusionproofovaloids}
If a set $C$ is $(FIO^m),$ (see Definition \ref{definitionscm3}), then either $C$ has nonempty interior or $C$ is a single point. In the case that $C$ is a singleton, say $C=\lbrace y_0 \rbrace,$ we can easily modify Sections \ref{subsectionproofcm1} and \ref{subsectionproofcm2} so as to construct a $C^m(\R^n)$ convex function $F$ such that $J^m_{y_0} F= P_{y_0},$ provided that the polynomial $P_{y_0}$ satisfies condition $(CW^m).$ We may thus suppose that $C$ has nonempty interior.

\medskip

Fix $m\in \N$ with $m \geq 3.$ Suppose that $C$ is $(FIO^{m-1})$ with nonempty interior and let $f\in C^m(\Rn)$ be a function satisfying $(CW^m)$ on $C.$ According to Definition \ref{definitionscm3}, we can write $C=\bigcap_{j=1}^N C_j$ so that, for each $1\leq j \leq N,$ there exist a number $ M_j >0$ and a function $\psi_j:\Rn \to \R$ of class $C^{m-1}(\Rn)$ such that $C_j=\psi_j^{-1}(-\infty,1]$ and $D^2\psi_j(x)(v^2) \geq M_j$ for all $x\in \Rn$ and $v\in \Sn.$ Let us denote $M= \min\lbrace M_j \: : \: 1 \leq j \leq N \rbrace.$ By Proposition \ref{stronglyconvexfunction}, for each $j\in \lbrace 1, \ldots, N \rbrace,$ the set $C_j$ is a convex compactum and there is a constant $\beta_j >0$ with $\psi_j(x)-1 \geq \beta_j d(x,C_j)$ whenever $x\notin C_j.$ Set $\beta = \min \lbrace \beta_j \: : \: 1 \leq j \leq N \rbrace.$ Using Proposition \ref{distanceintersection}, we obtain some $L>0$ such that $d(x,C) \leq L \max_{1\leq j \leq N} d(x,C_j)$ for all $x\in \Rn.$ In conclusion, we have found positive constants $L, \beta ,M$ satisfying
\begin{equation} \label{inequalitiesconstantsL}
d(x,C) \leq L \: \max_{1\leq j \leq N} d(x,C_j) \quad \text{for all} \quad x\in\R^n;
\end{equation}
\begin{equation} \label{inequalitiesconstantsbeta}
\psi_j(x)-1 \geq \beta d(x,C_j) \quad \text{for all} \quad x\notin C_j, \quad 1\leq j \leq N ;
\end{equation}
\begin{equation} \label{inequalitiesconstantsM}
D^2 \psi_j (x)(v^2) \geq M  \quad \text{for all} \quad x\in \Rn,\:v\in \Sn ,\quad 1\leq j \leq N.
\end{equation}
Let $\omega:[0,+\infty) \to [0,+\infty)$ be as in Lemma \ref{functionomega} and define the functions
$$
 g(t)  =  \left\lbrace 
	\begin{array}{ccl}
	\int_0^t \int_0^{t_2} \cdots \int_0^{t_{m-1}} \omega \left( 2^{m-2} s \right) ds \: dt_{m-1} \cdots dt_2 & \mbox{if } \: t>0 \\
	0  & \mbox{if } \:  t\leq 0,
	\end{array}
	\right.
$$
$$
h(t) = g (  L\beta^{-1} t ), \quad t\in \R,
$$
$$
\varphi(x)= \sum_{j=1}^N h ( \psi_j(x)-1), \quad x\in \Rn.
$$
It is clear that $g\in C^{m-1}(\R)$ with $g^{(k)}(0)=0$ for all $0 \leq k \leq m-1.$ By the definition of the $\psi_j$'s and $g,$ we have that $\varphi^{-1}(0)=C$ and $\varphi \in C^{m-1}(\Rn).$ It is routine to check that $\partial^\alpha \varphi (y)=0$ for all $y\in C$ and $|\alpha| \leq m-1,$ and therefore $J_y^{m-1} \varphi=0$ for all $y\in C.$ Now consider a point $x\in \R^n \setminus C$ and a direction $v\in \Sn.$ A simple calculation and the fact that $h'' \geq 0$ lead us to 
\begin{align*}
D^2\varphi (x)(v^2) & = \sum_{j=1}^N  h''(\psi_j(x)-1) \langle \nabla \psi_j(x),v \rangle ^2 + \sum_{j=1}^N  h'(\psi_j(x)-1) D^2 \psi_j(x)(v^2) \\
& \geq  \sum_{j=1}^N h'( \psi_j(x)-1)D^2 \psi_j(x)(v^2).
\end{align*}
We deduce from \eqref{inequalitiesconstantsM} that
$$
D^2\varphi (x)(v^2) \geq  \sum_{j=1}^N h'(\psi_j(x)-1) D^2 \psi_j(x)(v^2) \geq M \sum_{j=1}^N h'(\psi_j(x)-1).
$$
Now we consider an index $j:=j_x$ such that $d(x,C_j)= \max_{1\leq i \leq N} d(x,C_i).$ Observe that $x\notin C_j$ as $x \in \Rn \setminus C.$ This implies that $\psi_j(x)>1$ and therefore 
\begin{equation}\label{hessianvarphiwithindexj}
D^2\varphi (x)(v^2) \geq M h'(\psi_j(x)-1)= M L \beta^{-1} g' ( L \beta^{-1} (\psi_j(x)-1)).
\end{equation}
Using the inequalities \eqref{inequalitiesconstantsbeta} and \eqref{inequalitiesconstantsL} and the choice of $j,$ we obtain
$$
\psi_j(x)-1 \geq \beta d(x,C_j) \geq \beta L^{-1} d(x,C),
$$
which in turn implies
$$
g' ( L \beta^{-1} (\psi_j(x)-1)) \geq g ' (d(x,C)) =g'(t), \quad \text{where} \quad t:=d(x,C).
$$
Because $\omega$ is nonnegative and nondecreasing, we can easily write
\begin{align*}
g'(t) & \geq \int_{t/2}^t  \int_0^{t_2} \cdots \int_0^{t_{m-2}} \omega (2^{m-2}s)  ds \: dt_{m-2}\cdots dt_2 \\
& \geq \frac{t}{2}\int_0^{t/2} \int_0^{t_2}\cdots \int_0^{t_{m-3}} \omega (2^{m-2}s)  ds \: dt_{m-3}\cdots dt_2 \\
& \geq \frac{t}{2} \cdot \frac{t}{4} \cdots \frac{t}{2^{m-3}} \cdot \frac{t}{2^{m-2}} \omega (t)= \frac{t^{m-2}}{2^{1+2+3+\cdots+(m-2)}} \omega (t).
\end{align*} 
By plugging this estimation in \eqref{hessianvarphiwithindexj} we get
$$
D^2\varphi(x)(v^2) \geq  M L \beta^{-1} g'(t)= k(n,m,C) t^{m-2} \omega(t), \quad \text{where} \quad
k(n,m,C) =\frac{M L \beta^{-1}}{2^{1+2+3+\cdots+(m-2)}}.
$$
On the other hand, since $f:\Rn \rightarrow \R$ satisfies the condition $(CW^m)$ on $C,$ then the inequality of Lemma \ref{functionomega} holds for $f$ and thus
$$
D^2f(x)(v^2) \geq - \omega (t) t^{m-2}.
$$
Therefore $F:=f+\frac{2}{k(n,m,C)} \varphi$ has a strictly positive Hessian on $\Rn \setminus C,$ is of class $C^{m-1}(\Rn)$ and coincides with $f$ on $C.$ Moreover, since $J_y^{m-1} \varphi=0$ for all $y\in C,$ we have that $J_y^{m-1} F = J_y^{m-1}f$ for all $y\in C.$ Because $f$ is convex on $C$ and the extension $F$ is differentiable, we have that $F$ is convex on $\Rn.$ The proof of Theorem \ref{theoremspecialcase} is complete.

\bigskip

\section{Remarks and Counterexamples}\label{sectioncounterexamples}

The following example is a variation of \cite[Example 4]{SchulzSchwartz} and shows that our main result fails if we drop the assumption that $C$ be compact, even in the presence of strictly positive Hessians.
\begin{example}\label{modified bathtub}
{\em Let $C=\{(x,y)\in\R^2 \, : \,  x>0, \: xy\geq 1\}$ and define $$f(x,y)=-2\sqrt{xy} +\frac{1}{x+1} +\frac{1}{y+1}$$
for every $(x,y)\in C$. The set $C$ is convex and closed, with a nonempty interior, and it is routine to verify that $f$ has a strictly positive Hessian on $C$. We also have
$$
\nabla f(x,y)=\left( -x^{-\frac{1}{2}}y^{\frac{1}{2}}-\frac{1}{(x+1)^{2}} \, , \,\,  -x^{\frac{1}{2}}y^{-\frac{1}{2}}-\frac{1}{(y+1)^{2}}\right), \quad (x,y) \in C.
$$
We claim that $f$ does not have any convex extension to all of $\R^2$. In order to prove this it is sufficient to see that, for instance, $m(f)(-1,-1)=\infty$, where $m(f)$ is the minimal convex extension of $f$ defined in \eqref{definition of the minimal extension}. Considering the curve $\gamma(t)=(t, \frac{1}{t})$, $t>0$, which parameterizes the boundary of $C$, we have
$$
m(f)(-1,-1)\geq f(t, \tfrac{1}{t})+\langle \nabla f(t, \tfrac{1}{t}), (-1-t, -1-\tfrac{1}{t})\rangle = 2+t+\tfrac{1}{t},
$$
so by letting either $t\to\infty$ or $t\to 0^{+}$ we obtain $m(f)(-1,1)=\infty$. As a matter of fact, it is not difficult to see that $m(f)(x,y)=\infty$ for every $(x,y)\in\R^2$ such that $x<0$ or $y<0$.}
\end{example}

The following example shows that if $C$ has empty interior then one cannot expect to find smooth convex extensions (of functions satisfying $(W^m)$ and $(CW^m)$ on $C$) without experiencing a certain loss of differentiability. The example also shows that in $\R^2$ this loss amounts to at least two orders of smoothness, and that the situation does not improve as $m$ grows large (unless $m=\infty$, of course).

\begin{example}\label{CWm with m finite is not sufficient when intC is empty}
{\em Consider the function $\theta(y)= \frac{1-\cos(2\pi y)}{2\pi}, \: y\in \R.$ Clearly, $\theta \in C^\infty(\R)$, with $\theta(0)=\theta(1)=0, \: \theta(1/2)= \frac{1}{\pi}$ and $\theta'(y)= \sin (2\pi y).$ Define $h(x,y)=\theta(y) x^m, \: (x,y) \in \R^2.$ Let $C:= \lbrace 0 \rbrace \times [0,1].$ We have $D^k h=0$ on $C$ for all $k\in \lbrace 0, \ldots,m-1 \rbrace$, and
$$
D^m h(x,y)=m! \theta(y) \overbrace{e_1^* \otimes \cdots \otimes e_1^*}^{m} \textrm{ for }  (x,y) \in C
$$
(here $e_1^*$ denotes the linear function $(x_1, x_2)\mapsto x_1$).
Therefore $D^m h(0,0)=D^m h (0,1) = 0$, and
$
D^m h (0, \tfrac{1}{2} ) = \frac{m!}{\pi } e_1^* \otimes \cdots\otimes e_1^*
$.
We claim that if $m\geq 2$ is even, then there is no convex function $F\in C^m(\R^2)$ such that $D^k F = D^k h$ on $C$ for $k\in \lbrace 0, \ldots,m  \rbrace $. We also claim that $h$ satisfies conditions $(W^{\infty})$ and $(CW^{m+1})$ (and in particular $(CW^m)$ too) on $C$.}
\end{example}
The first claim immediately follows from the following.
\begin{remark} \label{lema no existencia}
{\em If $m \geq 2$, there exists no convex function $f\in C^m (\R^2)$ such that $D^k f(0,y)=0$ for all $k\in \lbrace 0, \ldots, m-1 \rbrace$, $y\in [0,1]$, and such that $D^m f(0,0) =  D^mf(0,1)=0$ and $D^m f(0, \frac{1}{2} ) = A e_1^* \otimes \cdots \otimes e_1^*$, where $A>0$ is a constant.}
\end{remark}
\begin{proof}
For the sake of contradiction, suppose there is such an $f$. Using Taylor's theorem we have
$$
f(x,y)= \frac{1}{m!} D^mf (0,y_0) ( x, y-y_0)^m + R(x,y,y_0) \quad (x,y) \in \R^2, \: y_0 \in [0,1],
$$
 where $\dfrac{R(x,y)}{|(x,y-y_0)|^m} \to 0$ as $(x,y) \to (0,y_0)$, uniformly on $y_0 \in [0,1].$ Fix $0 < \varepsilon < \frac{A}{2 m!}$, and take $\delta=\delta(\varepsilon) >0$ such that if $y_0 \in [0,1]$ and $(x,y) \in \R^2$ satisfy $ ( x^2 +(y-y_0)^2 )^{1/2} \leq \delta$ then
$$
\big| f(x,y)-\frac{1}{m!} D^mf (0,y_0) ( x, y-y_0)^m \big| =  |R(x,y)| \leq \varepsilon ( x^2 +(y-y_0)^2 )^{\frac{m}{2}}.
$$
Evaluating for $y=y_0=1/2$ we obtain
$$
\Big| f(x,\tfrac{1}{2})- A \frac{x^m}{m!} \Big| \leq \varepsilon |x|^m \quad \textrm{whenever} \quad |x| \leq \delta.
$$
For $y=y_0 \in \lbrace 0,1 \rbrace $ and $|x| \leq \delta$ we get
$$
\max \lbrace |f(x,0)|, |f(x,1)| \rbrace \leq \varepsilon |x|^m.
$$
Fix $x_0 >0$ with $x_0 \leq \delta.$  We then have
$$
f(x_0, \tfrac{1}{2}) \geq A \frac{ x_0^m}{m!}- \varepsilon x_0^m > 2\varepsilon x_0^m -\varepsilon x_0^m= \varepsilon x_0^m \geq \max \lbrace f(x_0,0), f(x_0,1) \rbrace.
$$
This implies that $[0,1] \ni t \mapsto \varphi(t) = f(x_0,t)$ satisfies $\varphi(\tfrac{1}{2}) > \frac{1}{2} \varphi(0) +\frac{1}{2} \varphi(1),$ and in particular $f$ cannot be convex.
\end{proof}

Let us now prove our second claim. It is obvious that $h$ satisfies $(W^{k})$ for every $k$. We only have to check that $h$ satisfies $(CW^{m+1})$ on $C$. We must see that, given $\varepsilon >0,$ there exists $t_\varepsilon>0$ such that
\begin{align*}
Q_{m+1}(y,t,v,w):= \frac{\frac{1}{(m-2)!} D^m h(0,y)(v^2 ,w^{m-2}) +\frac{t}{(m-1)!}D^{m+1}h(0,y)(v^2,w^{m-1})  }{t} \geq -\varepsilon,
\end{align*}
for every $y\in [0,1], \: v,w\in \mathbb{S}^{1}, \: 0< t \leq t_\varepsilon.$
It is not difficult to check that
\begin{align*}
D^{m+1}h(0,y)(v^2,w^{m-1}) & = \frac{\partial^{m+1} h}{\partial x^m \partial y}(0,y) \left[ (m-1) v_1^2 w_1^{m-2} w_2 + 2 v_1v_2 w_1^{m-1} \right] \\
& = m! \: \theta'(y) \left[ (m-1) v_1^2 w_1^{m-2} v_2 + 2 v_1v_2 w_1^{m-1} \right].
\end{align*}
On the other hand $D^m h(0,y)(v^2 ,w^{m-2})= m! \: \theta(y) v_1^2 w_1^{m-2}$. Let us fix $t_{\varepsilon}$ such that
$$
0<t_{\varepsilon}\leq \min\left( 1, \, \frac{\varepsilon}{4\pi (2m+3)(m+1)m(m-1)}\right) .
$$
Take $y\in [0,1], \: v,w\in \mathbb{S}^{1}$ and $0< t \leq t_\varepsilon.$ We have
$$
Q_{m+1}(y,t,v,w)= \frac{1}{t} \left[ \frac{m!}{(m-2)!} \theta(y)  v_1^2 w_1^{m-2} + \frac{m!}{(m-1)!}  t \theta'(y) \left( (m-1) v_1^2 w_1^{m-2} w_2 + 2 v_1v_2 w_1^{m-1} \right) \right].
$$
Since $m$ is even, we have $w_1^{m-2} \geq 0$, and it is easily seen that
\begin{align}\label{desigualdad principal}
& Q_{m+1}(y,t,v,w)  \geq  \frac{m(m-1)|v_1||w_1|^{m-2}}{t} \left( \theta(y) |v_1| -(m+1) t |\theta'(y)| \right).
\end{align}
Let us now distinguish the following cases.

\noindent\textbf{{Case 1.}} Assume $y\in [\frac{1}{4}, \frac{3}{4} ].$ Then $\cos(2\pi y) \leq 0 ,$ which implies $\theta(y) \geq \frac{1}{2\pi}.$ Since we always have $|\theta '(y)| = | \sin(2\pi y) | \leq 1$, it follows from (\ref{desigualdad principal}) that
$$
Q_{m+1}(y,t,v,w)\geq \frac{m(m-1)|v_1||w_1|^{m-2}}{t} \left( \frac{|v_1|}{2\pi} -(m+1) t \right).
$$
\noindent \textbf{{Subcase 1.1.}} Assume $ |v_1| \geq 2\pi (m+1)t.$ Then it is clear that $Q_{m+1}(y,t,v,w) \geq 0 \geq -\varepsilon.$

\noindent \textbf{{Subcase 1.2.}} Assume $ 2\pi (m+1) t^2 \leq |v_1| \leq 2\pi (m+1)t.$ Then, since $|w_1|, t , 1-t \leq 1$, we obtain
\begin{align*}
&Q_{m+1}(y,t,v,w) \geq \frac{m(m-1)|v_1||w_1|^{m-2}}{t} \left( (m+1) t^2 -(m+1) t \right)\\
&  = (m+1)m (m-1) |v_1| | w_1|^{m-2} (t-1) \geq - 2\pi (m+1)^2m(m-1)t |w_1|^{m-2} (1-t) \\
& \geq - 2\pi t (m+1)^2 m(m-1) \geq - 2\pi t_\varepsilon (m+1)^2 m(m-1) \geq -\varepsilon.
\end{align*}
\noindent\textbf{{Subcase 1.3.}} Assume $ |v_1| \leq 2\pi (m+1)t^2.$ We have
\begin{align*}
&Q_{m+1}(y,t,v,w)  \geq -\frac{m(m-1)|v_1||w_1|^{m-2}}{t} \left( (m+1) t -\frac{|v_1|}{2\pi} \right)\\
&  \geq -\frac{2\pi m (m-1) (m+1)t^2 | w_1|^{m-2}}{t} \left( m+1+ \frac{1}{2\pi} \right)\geq - 2\pi (m+1)m(m-1) (m+2) \:t \\
& \geq - 2\pi (m+1)m(m-1) (m+2) \:t_\varepsilon \geq -\varepsilon.
\end{align*}
\noindent\textbf{{Case 2.}} Assume $y\in [0, \frac{1}{4} ).$ Then  $2\pi y \in [0,\frac{\pi}{2})$, and we have
$$
\theta(y) = \frac{1-\cos(2\pi y)}{2\pi} = \frac{\sin^2(\pi y)}{\pi}.
$$
On the other hand,
$$
|\theta'(y)| = |\sin(2\pi y)| = \sin (2\pi y) = 2 \sin (\pi y) \cos (\pi y).
$$
By plugging this identity in (\ref{desigualdad principal}), we get
\begin{align*}
Q_{m+1}(y,t,v,w) &\geq \frac{m(m-1) |v_1||w_1|^{m-2}}{t} \left( \frac{\sin^2(\pi y)|v_1|}{\pi}- (m+1) \sin (2\pi y) t \right) \\
& = \frac{m(m-1) |v_1||w_1|^{m-2} \sin (\pi y)}{t} \left( \frac{\sin(\pi y)|v_1|}{\pi}- 2(m+1) \cos (\pi y) t \right)
\end{align*}
\noindent\textbf{{Subcase 2.1.}} Assume $ \sin(\pi y)|v_1| \geq 2\pi (m+1) \cos (\pi y) t.$ Then obviously $Q_{m+1}(y,t,v,w) \geq 0 \geq -\varepsilon.$

\noindent\textbf{{Subcase 2.2.}} Assume $2\pi (m+1) \cos (\pi y) t^2 \leq  \sin(\pi y)|v_1| \leq 2\pi (m+1) \cos (\pi y) t.$ We have
$$
Q_{m+1}(y,t,v,w)  \geq m (m-1) |v_1| |w_1|^{m-2} \sin ( \pi y) 2 (m+1) \cos ( \pi y) (t-1),
$$
whose absolute value is smaller than or equal to
\begin{align*}
&2(m+1)m(m-1) |v_1| \sin (\pi y)  \leq 4 \pi (m+1)^2 m(m-1)  \cos(\pi y) t \\
& \leq 4 \pi (m+1)^2 m(m-1)t \leq 4 \pi (m+1)^2 m(m-1)t_\varepsilon \leq \varepsilon.
\end{align*}
This shows that $Q_{m+1}(y,t,v,w) \geq -\varepsilon.$

\noindent\textbf{{Subcase 2.3.}} Assume $\sin(\pi y)|v_1|  \leq 2\pi (m+1) \cos (\pi y) t^2.$ Recall that
$$
Q_{m+1}(y,t,v,w) \geq \frac{m(m-1) |v_1||w_1|^{m-2} \sin (\pi y)}{t} \left( \frac{\sin(\pi y)|v_1|}{\pi}- 2(m+1) \cos (\pi y) t \right).
$$
The modulus of the last term is less than or equal to
\begin{align*}
& \frac{m(m-1)|v_1| \sin (\pi y) \left( \frac{1}{\pi} + 2(m+1) \right)}{t} \leq \frac{m(m-1) 2\pi (m+1) \cos(\pi y) t^2(1 + 2(m+1))}{t} \\
& \leq 2\pi (m+1) m(m-1)(2m+3) t \leq 2\pi(2m+3) (m+1) m(m-1) t_\varepsilon \leq \varepsilon.
\end{align*}
Hence $Q_{m+1}(y,t,v,w) \geq -\varepsilon.$

\noindent\textbf{{Case 3.}} Assume finally that $y\in (\frac{3}{4}, 1].$ Take $z=1-y.$ Clearly $\cos(2\pi z) = \cos(2\pi y)$, and $\sin(2\pi z)= -\sin(2\pi y).$ Therefore $\theta(z)=\theta(y)$ and $|\theta'(y)|=|\theta'(z)|$, hence
\begin{align*}
Q_{m+1}(y,t,v,w) &\geq \frac{m(m-1)|v_1||w_1|^{m-2}}{t} \left( \theta(y) |v_1| -(m+1) t |\theta'(y)| \right) \\
& = \frac{m(m-1)|v_1||w_1|^{m-2}}{t} \left( \theta(z) |v_1| -(m+1) t |\theta'(z)| \right),
\end{align*}
and since $z\in [0,\frac{1}{4}),$ we can apply Case 2 with $z$ instead of $y$ to obtain $Q_{m+1}(y,t,v,w) \geq -\varepsilon.$

\medskip



\begin{thebibliography}{}

\bibitem{AM}
D. Azagra, C. Mudarra, {Whitney extension theorems for convex functions of the classes $C^1$ and $C^{1, \omega}$}. Proc. London Math. Soc. (3) 114 (2017), 133-158.


\bibitem{ALM}
D. Azagra, E. Le Gruyer, C. Mudarra, {\em Explicit formulas for $C^{1,1}$ and $C^{1, \omega}_{\textrm{conv}}$ extensions of 1-jets in Hilbert and superreflexive spaces}. J. Funct. Anal. 274 (2018), 3003-3032.

\bibitem{AM2}
D. Azagra, C. Mudarra, {\em Global geometry and $C^1$ convex extensions of 1-jets}.  Analysis and PDE 12 (2019) no. 4, 1065-1099.


\bibitem{AFPAMS2002}
D. Azagra, J. Ferrera, {\em Every closed convex set is the set of minimizers of some $C^{\infty}$ smooth convex function}.
Proc. Amer. Math. Soc. 130 (2002), no. 12, 3687--3692.

\bibitem{BierstoneMilmanPawluka1}
E. Bierstone, P. Milman, W. Pawluka, {\em Differentiable functions defined in closed sets. A problem of Whitney}. Invent. Math. 151 (2003), 329--352.

\bibitem{BierstoneMilmanPawluka2}
E. Bierstone, P. Milman, W. Pawluka, {\em Higher order tangents and Fefferman's paper on Whitney's extension problem}. Ann. Math. 164 (2006), 361--370.

\bibitem{BorweinMontesinosVanderwerff}
J. Borwein, V. Montesinos, J. Vanderwerff, {\em Boundedness, differentiability and extensions of convex functions}. J. Convex Anal. 13 (2006), 587--602.


\bibitem{BrudnyiBrudnyi}
A. Brudnyi, Y. Brudnyi, {\em Methods of geometric analysis in extension and trace problems. Volumes 1 and 2.}
Monographs in Mathematics, 102 and 103. Birkh\"auser/Springer Basel AG, Basel, 2012.


\bibitem{BrudnyiShvartsman}
Y. Brudnyi, P. Shvartsman, {\em Whitney's extension problem for multivariate $C^{1, \omega}$-functions}. Trans. Am. Math. Soc. 353 (2001), 2487--2512.

\bibitem{BucicovschiLebl}
O. Bucicovschi, J. Lebl, {\em On the continuity and regularity of convex extensions}.
J. Convex Anal. 20 (2013), no. 4, 1113--1126.

\bibitem{Federer}
H. Federer, {\em Geometric measure theory}. Springer-Verlag New York Inc., New York, 1969.

\bibitem{FeffermanAnnals2005}
C. Fefferman, {\em A sharp form of Whitney's extension theorem}. Ann. of Math. (2) 161 (2005), no. 1, 509--577.

\bibitem{FeffermanAnnals2006}
C. Fefferman, {\em Whitney's extension problem for $C^{m}$}. Ann. of Math. (2) 164 (2006), no. 1, 313--359.

\bibitem{FeffermanBullAMS2009}
C. Fefferman, {\em Whitney's extension problems and interpolation of data}. Bull. Amer. Math. Soc. (N.S.) 46 (2009), no. 2, 207--220.

\bibitem{FeffermanIsraelLuli}
C. Fefferman, A. Israel, G.K. Luli, {\em Sobolev extension by linear operators}. J. Amer. Math. Soc. 27 (2014), no. 1, 69--145.

\bibitem{GhomiJDG2001}
M. Ghomi, {\em Strictly convex submanifolds and hypersurfaces of positive curvature}. J. Differential Geom. 57 (2001), 239--271.

\bibitem{GhomiPAMS2002}
M. Ghomi, {\em The problem of optimal smoothing for convex functions}. Proc. Amer. Math. Soc. 130 (2002) no. 8, 2255--2259.

\bibitem{GhomiBLMS2004}
M. Ghomi, {\em Optimal smoothing for convex polytopes}. Bull. London Math. Soc. 36 (2004), 483--492.

\bibitem{Glaeser}
G. Glaeser, {\em Etudes de quelques alg\`ebres tayloriennes}. J. d'Analyse 6 (1958), 1-124.

\bibitem{GriewankRabier}
A. Griewank, P.J. Rabier, {\em On the smoothness of convex envelopes}. Trans. Amer. Math. Soc. 322 (1990) 691--709.

\bibitem{Israel}
A. Israel, {\em A bounded linear extension operator for $L^{2}_p(\mathbb{R}^{2})$}. Ann. of Math. (2) 178 (2013), no. 1, 183--230.

\bibitem{KirchheimKristensen}
B. Kirchheim, J. Kristensen, {\em Differentiability of convex envelopes}. C. R. Acad. Sci. Paris S\'er. I Math. 333 (2001), no. 8, 725--728.

\bibitem{Rockafellar}
T. Rockafellar, {\em Convex Analysis}. Princeton Univ. Press, Princeton, NJ, 1970.

\bibitem{ShvartsmanAdvances2009}
P. Shvartsman, {\em Sobolev $W^{1,p}$ spaces on closed subsets of $\R^{n}$}. Adv. Math. 220 (2009) 1842--1922

\bibitem{ShvartsmanAdvances2014}
P. Shvartsman, {\em Sobolev $L^{2}_p$-functions on closed subsets of $\mathbb{R}^{2}$}. Adv. Math. 252 (2014), 22--113.

\bibitem{SchulzSchwartz}
K. Schulz, B. Schwartz, {\em Finite extensions of convex functions}.
Math. Operationsforsch. Statist. Ser. Optim. 10 (1979), no. 4, 501--509.

\bibitem{Stein}
E. Stein, {\em Singular integrals and differentiability properties of functions}. Princeton, University Press, 1970.

\bibitem{VeselyZajicek}
L. Vesel\'y, L. Zaj\'icek, {\em On extensions of d.c. functions and convex functions}. J. Convex Anal. 17 (2010), no. 2, 427--440.

\bibitem{Whitney}
H. Whitney, {\em Analytic extensions of differentiable functions defined in closed sets}. Trans. Amer. Math. Soc. 36 (1934), 63--89.

\bibitem{MinYan}
M. Yan, {\em Extension of Convex Function}. J. Convex Anal. 21 (2014) no. 4, 965--987.


\end{thebibliography}
\end{document}